\documentclass[a4paper,12pt]{article}
\usepackage[utf8]{inputenc}
\usepackage{amsmath,amssymb,mathrsfs,amscd,graphicx,multicol,amsthm}
\usepackage{subcaption}
\usepackage{tikz}
\usepackage{float}
\usepackage{verbatim}
\usepackage{epstopdf}
\usepackage{titletoc}
\usepackage{hyperref }
\numberwithin{equation}{section}
\allowdisplaybreaks

\theoremstyle{definition}
\newtheorem{theorem}{Theorem}[section]

\newtheorem{Proposition}[theorem]{Proposition}
\newtheorem{lemma}[theorem]{Lemma}
\newtheorem{corollary}[theorem]{Corollary}
\newtheorem{RHP}{RH problem}[section]
\newtheorem{remark}[theorem]{Remark}

\DeclareMathOperator*{\Res}{Res}

\def\C{\mathbb{C}}
\def\L{\mathbb{L}}
\def\R{\mathbb{R}}

\def\dd{\mathrm{d}}
\def\ii{\mathrm{i}}
\def\ee{\mathrm{e}}

\begin{document}
	\title{\bf Long-time asymptotics of the KdV equation with delta function initial profiles}
	
	\author{Xuliang Liu$^\text{a}$, Deng-Shan Wang$^\text{b}$}
	\date{}
	\maketitle
	$^{\rm a)}$~Laboratory of Mathematics and Complex Systems (Ministry of Education),  \\ School of Mathematical Sciences, Beijing Normal University, Beijing 100875, China.
	
	E-mail: xuliangliu@mail.bnu.edu.cn
	
	$^{\rm b)}$~Laboratory of Mathematics and Complex Systems (Ministry of Education),  \\ School of Mathematical Sciences, Beijing Normal University, Beijing 100875, China. 
	
	E-mail: dswang@bnu.edu.cn 

	\begin{abstract}
		
		This work investigates the long-time asymptotic behaviors of the solution to the KdV equation with delta function initial profiles in different regions, employing the Riemann-Hilbert formulation and Deift-Zhou nonlinear steepest descent method. When the initial value is a delta potential well, the asymptotic solution is predominantly dominated by a single soliton in certain region for $x>0$, while in other regions, the dispersive tails including self-similar region, collisionless shock region and dispersive wave region, play a more significant role. Conversely, when the initial value is a delta potential barrier, the soliton region is absent, although the dispersive tails still persist. Moreover, the general delta function initial profile with $L$-spikes is also studied and it is proved that one to $L$ solitons will be generated in soliton region, which depends on the sizes of the distance and height of the spikes. The leading-order terms of the solution in each region are derived, highlighting the efficacy of the Riemann-Hilbert formulation in elucidating the long-time behaviors of integrable systems.

	\end{abstract}
	
	\vspace{8pt}
	{\footnotesize
	}
	
	\tableofcontents
	
	\section{Introduction}
	
	The Korteweg-de Vries (KdV) equation  is one of the most fundamental nonlinear partial differential equations in mathematical physics. Originally derived in 1895 by D. J. Korteweg and his student G. de Vries \cite{KdV-1895} to describe the propagation of shallow water waves, the KdV equation has a wide range of applications in many fields, such as plasma physics, nonlinear optics and fluid dynamics. Its significance lies not only on its physical relevance but also in its mathematical structure: the KdV equation is a prototypical example of an integrable system that possess an infinite number of conserved quantities and can be solved exactly using the inverse scattering transform (IST), which was discovered by Gardner, Greene, Kruskal and Miura in 1967 \cite{GGKM-1965}, providing a powerful framework for analyzing exact soliton solutions and long-time asymptotic behaviors of integrable systems. The IST method is a nonlinear analogue of the Fourier transform, which is used to solve linear partial differential equations. It relies on the association of the nonlinear equation with a linear eigenvalue problem, i.e., Lax pair \cite{Lax-1968}, typically the Schr\"{o}dinger equation \cite{Drazin-1989,Marchenko-2011,Egorova-Teschl-2015} in the case of the KdV equation. The initial data of the integrable system are used to determine the scattering data, which includes the eigenvalues, reflection coefficient and normalization constants associated with the linear eigenvalue problem. The solutions of the integrable systems are then reconstructed from the scattering data by using a set of integral equations. This method \cite{GGKM-1965} not only provides exact solutions but also offers deep insights into the qualitative behavior of the solutions, such as soliton formation, approximated interaction solution of soliton \cite{Ablowitz-Segur-B-1981} and mean field \cite{Ablowitz-Luo-2018,Ablowitz-Luo-2023} and long-time asymptotics \cite{Ablowitz-Kodama-1982}-\cite{Girotti-Grava-2021}.
	\par
	The study of long-time asymptotic behaviors of solutions to the KdV equation has been a central topic in the field of integrable systems. In the earlier years, Hruslov \cite{Hruslov-1976} studies the asymptotics of solution to KdV equation with initial data of step type by IST. Subsequently, Cohen \cite{Cohen-1984} again discussed the asymptotic solution of the KdV equation with step-like initial profile. One of the key developments in this area was attributed to Deift and Zhou \cite{Deift-Zhou-1993}, who established the Riemann-Hilbert (RH) problem and proposed the Deift-Zhou nonlinear steepest descent method to give a rigorous framework for analyzing the long-time behavior of solutions to integrable nonlinear equations, even the theory of random matrix models and integrable statistical mechanics \cite{Deift-2017}. The RH approach is particularly well-suited for studying the asymptotics of solutions in different regions of the upper $(x,t)$-plane, where the behaviors of the solution can vary significantly. In their seminal work, Deift, Venakides and Zhou \cite{Deift-Venakides-Zhou-1994} applied the RH approach to the KdV equation with smooth, decaying initial data, showing that the long-time asymptotic behavior of the solution could be described in terms of a finite number of solitons and a dispersive tail. The solitons are localized, traveling wave solutions that emerge from the initial data and propagate without changing their shape, while the dispersive tail decays as $t^{-1/2}$ and exhibits oscillatory behavior. The RH method allows for a detailed analysis of the transition between these different regions, providing precise estimates for the solution in each case.
	\par
	The work of Deift, Venakides and Zhou \cite{Deift-Venakides-Zhou-1994} has inspired a wealth of research on the long-time asymptotics of integrable systems. In particular, Grunert and Teschl \cite{Grunert-Teschl-2009} investigated the long-time asymptotic behavior of the KdV equation with decaying initial data in the soliton and similarity regions. Then Teschl and his colleagues \cite{Teschl-Nonlinearity-2013} conducted a detailed study on the long-time asymptotics of the KdV equation with step-like initial data, providing an in-depth description of the solution in the elliptic wave region. Subsequently, they \cite{Teschl-JDE-2016} studied rarefaction waves in the KdV equation with step-like initial data via Deift-Zhou nonlinear steepest descent method. Bilman, Trogdon, Olver and Deconinck \cite{Trogdon-Olver-Deconinck-2012,Trogdon-Deconinck-2014,Bilman-Trogdon-2020} developed numerical inverse scattering method to analyze the long-time behaviors of the KdV equation. These studies provide a detailed description of the soliton parameters and the decay of the dispersive tail, offering new insights into the dynamics of the KdV equation with singular initial data.
	\par
	This paper builds on the former work on asymptotic analysis of the KdV equation with various initial data \cite{Grunert-Teschl-2009,Teschl-Nonlinearity-2013,Teschl-JDE-2016} to explore the long-time asymptotic behaviors of the KdV equation with delta function initial profiles \cite{Drazin-1989}. Firstly, the KdV equation with only one delta function initial profile is considered
	\begin{equation}
		\label{InitialProblem}
		\left\{
		\begin{aligned}
			&u_t-6uu_x+u_{xxx}=0,\\
			&u(x,0)=u_0(x) = -U_0\delta(x),
		\end{aligned}
		\right.
	\end{equation}
	where $U_0\in\R\backslash\{0\}$ is constant, and $\delta(x)$ is the delta function which satisfies
	\begin{equation}
		\delta(x)=\begin{cases}
			\infty,&x=0,\\
			0,&x\neq0.
		\end{cases}\quad\quad
		\int_{-\infty}^{\infty}\delta(x)\dd x=1.
	\end{equation}
	Notice that the case of $U_0>0$ corresponds to a delta potential well, whereas the case of $U_0<0$ corresponds to a delta potential barrier.
	\par
	More generally, we also study the long-time asymptotic behaviors of KdV equation with general delta function initial profile that features multiple spikes \cite{Cohen-Kappeler-1987}
	\begin{equation}
		\label{L-delta-intial}
		u(x,0)=-\sum_{n=1}^LU_n\delta_n(x),
	\end{equation}
	where $U_n\in\R\backslash\{0\}$ and $\delta_n(x)=\delta(x-x_n)$ with $x_n\in \R$. Specifically, it reduces to the one delta function initial profile in (\ref{InitialProblem}) for $L=1$, $U_1=U_0$ and $x_1=0$.
	\par
	It is known that the KdV equation is a completely integrable system associated with the linear Schr\"{o}dinger equation $\mathcal{L}\phi=\lambda \phi$ with $\mathcal{L}=\frac{\partial^2}{\partial x^2}-u$ in the context of inverse scattering theory \cite{Atkinson-1975,Aktosun-1999}. More than fifty years ago, there has been much work on the direct/inverse scattering theory of the Schr\"{o}dinger operator \cite{Buslaev-Fomin-1962,Cohen-Kappeler-1985} with various initial condition such as step-like initial values. The study of the linear Schr\"{o}dinger equation with delta function potential is a fundamental topic in quantum mechanics and mathematical physics, offering deep insights into the behavior of quantum systems with singular interactions. The delta function potential is significant as it offers a simple yet powerful framework for modeling point-like interactions, such as crystal lattice impurities or quantum well interfaces. In one-dimensional case, the Schr\"{o}dinger equation with a delta function potential can be solved analytically, and the solutions describe bound states and scattering states, depending on the sign and magnitude of $U_0$. In 2010, Deift and Park \cite{Deift-Park-2011} explored the long-time asymptotics of solutions to the nonlinear Schr\"{o}dinger equation with a delta function potential.
	\par
	In addition, it is well-known that the KdV equation is related with the modified-KdV (mKdV) equation \cite{Deift-Zhou-1993}
	\begin{equation}
		\label{mKdV}
		v_t-6v^2v_x+v_{xxx}=0,
	\end{equation}
	through the	Miura transformation \cite{Miura-1968}
	\begin{equation}\label{Miura-mKdV}
		u=v^2+v_x.
	\end{equation}
	Therefore, to some extent, one can study the long-time asymptotic behavior of the KdV equation through that of the mKdV equation. When investigating the long-time asymptotics of the initial problem (\ref{InitialProblem}) in Painlev\'e region, the Miura transformation is taken to simplify the calculation.
	However, directly studying the KdV equation will reveal more abundant phenomena.
	\par
	This work adopts the IST method \cite{GGKM-1965} combined with the Deift-Zhou nonlinear steepest descent method \cite{Deift-Zhou-1993} to analyze the asymptotic solution of initial problem (\ref{InitialProblem}) in different regions of the $(x,t)$-plane. It will be shown that although the delta-function initial profile in (\ref{InitialProblem}) is discontinuous and contains a singularity at the origin, the solution to this problem reveals surprisingly rich and regular characteristics \cite{Cohen-1979} in its long-time asymptotic behavior.
	\par
	This paper is organized as follows: in Section 2, the main results of this paper is proposed. In Section 3, the IST for the initial problem (\ref{InitialProblem}) is reviewed and the corresponding RH problem is formulated. Section 4 applies the Deift-Zhou nonlinear steepest descent method to analyze the long-time behavior of the solution to the initial problem (\ref{InitialProblem}) for the case of $U_0>0$ and $U_0<0$, respectively.

	\section{Main Results}
	
	This section lists the main results of the present work. For the initial value problem $(\ref{InitialProblem})$, one can derive the reflection coefficient
	\begin{equation}\label{r-k-one-delta}
		r(k)=\frac{\ii U_0}{\ii U_0-2k},
	\end{equation}
	which has only one simple pole and is analytic on the complex plane $\C$ except for $k=\frac{\ii U_0}{2}$. Based on this fact, one can omit the step of the analytic continuation of reflection coefficient $r(k)$ in the process of nonlinear steepest descent analysis.
	\par	
	Through deforming and analyzing the RH problem associated with the initial value problem (\ref{InitialProblem}), the long-time asymptotics of the solution is derived in different regions. In what follows, the main results of this work are proposed by classifying the sign of the parameter $U_0$, i.e., $U_0>0$ and $U_0<0$. The long-time asymptotic behaviors of the solution to initial value problem (\ref{InitialProblem}) in the different regions of the upper $(x, t)$-half plane are shown in Figure \ref{x0t-plane}, where the ``Region T" in the second quadrant represents the transition region that was firstly conjectured for KdV equation with Schwartz class initial condition conjectured in Ref. \cite{Trogdon-Olver-Deconinck-2012} by numerical inverse scattering approach. Here we believe that the “Region T" also appears in the case of KdV equation with delta-function initial profle. To the best of our knowledge,  the long-time asymptotics is not known in the ``Region T" until now, but Bilman, Deift, Kriecherbauer, Mclaughlin and Nenciu \cite{Bilman-Deift-Kriecherbauer-Mclaughlin-Nenciu-2024} are working on this topic now.

	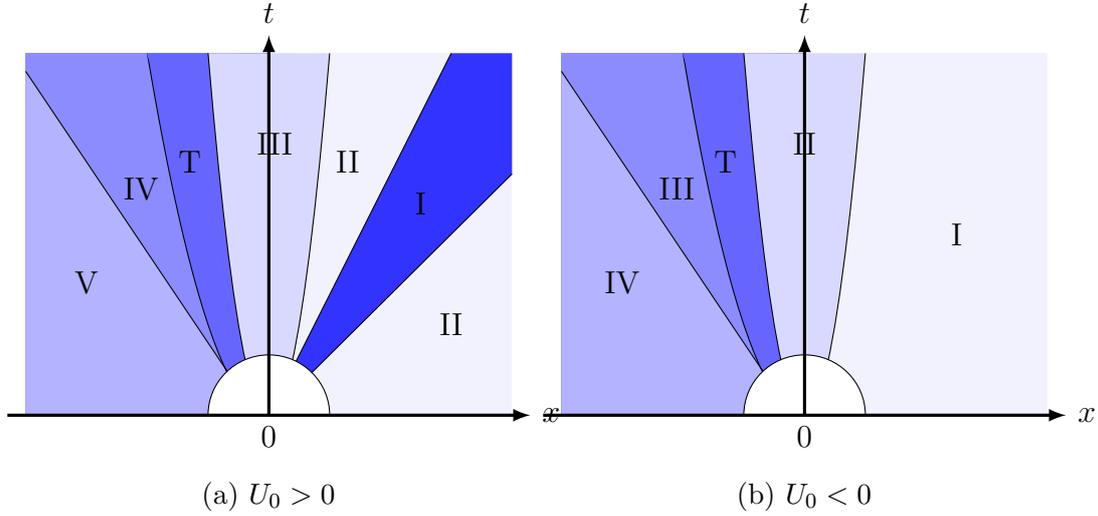
\begin{figure}[htpb]
		\centering
		\begin{subfigure}{0.45\textwidth}
			\begin{tikzpicture}[scale=0.8]
				\fill[blue!5] (0,0)--(4,0)--(4,6)--(0,6);
				\fill[blue!80](0,0)--(4,4)--(4,6)--(3,6);
				\fill[blue!45](-0,0)--(-4,0)--(-4,6)--(0,6);
				\fill[blue!30](-0.2,0)--(-4,0)--(-4,5.7);
				\draw[-](0,0)--(4,4);
				\draw[-](0,0)--(3,6);
				\draw[domain=-2:0,fill=blue!60]plot({\x},{1.5*\x*\x});
				\fill[blue!60](0,0)--(0,6)--(-2.01,6);
				\draw[domain=-1:1,fill=blue!15]plot({\x},{6*\x*\x});
				\draw[-](-0.2,0)--(-4,5.7);
				\node[]at(-2.1,3.75)[]{IV};
				\node[]at(-1.3,4.2)[]{T};
				\node[below]at(0,0){$0$};
				\node[]at(3,1.5)[]{II};
				\node[]at(1.3,4.2)[]{II};
				\node[]at(2.5,3.5)[]{I};
				\node[]at(0.1,4.5)[]{III};
				\node[]at(-3,2.2)[]{V};
				\fill[white](-1,0)--(1,0)arc[start angle=0,end angle=180,radius=1]--cycle;
				\draw (1,0) arc[start angle=0, end angle=180, radius=1];
				\draw[very thick,-latex] (-4.3,0) -- (4.3,0) node[right] {$x$};
				\draw[very thick,-latex] (0,0) -- (0,6.3) node[above] {$t$};
			\end{tikzpicture}
			\caption{$U_0>0$}
			\label{x0t-plane-U0>0}
		\end{subfigure}
		\begin{subfigure}{0.45\textwidth}
			\begin{tikzpicture}[scale=0.8]
				\fill[blue!5] (0,0)--(4,0)--(4,6)--(0,6);
				\fill[blue!45](-0,0)--(-4,0)--(-4,6)--(0,6);
				\fill[blue!30](-0.2,0)--(-4,0)--(-4,5.7);
				\draw[domain=-2:0,fill=blue!60]plot({\x},{1.5*\x*\x});
				\fill[blue!60](0,0)--(0,6)--(-2.01,6);
				\draw[domain=-1:1,fill=blue!15]plot({\x},{6*\x*\x});
				\draw[-](-0.2,0)--(-4,5.7);
				\node[]at(-2.1,3.75)[]{III};
				\node[]at(-1.3,4.2)[]{T};
				\node[below]at(0,0){$0$};
				\node[]at(2.5,3)[]{I};
				\node[]at(0,4.5)[]{II};
				\node[]at(-3,2.2)[]{IV};
				\fill[white](-1,0)--(1,0)arc[start angle=0,end angle=180,radius=1]--cycle;
				\draw (1,0) arc[start angle=0, end angle=180, radius=1];
				\draw[very thick,-latex] (-4.3,0) -- (4.3,0) node[right] {$x$};
				\draw[very thick,-latex] (0,0) -- (0,6.3) node[above] {$t$};
			\end{tikzpicture}
			\caption{$U_0<0$}
			\label{x0t-plane-U0<0}
		\end{subfigure}
		\caption{Regions for the long-time asymptotics of the KdV equation.}
		\label{x0t-plane}	
	\end{figure}
	
	\begin{theorem}
		\label{theorem-1} \rm
		For $U_0>0$, let $\epsilon>0$ be small enough and $l>0$ be an integer, then as $t\to \infty$, the long-time asymptotic behaviors of the initial value problem $(\ref{InitialProblem})$ are formulated in the following forms (see Figure \ref{x0t-plane-U0>0})
		\begin{itemize}
			\item[(i)] In the soliton region (Region I), i.e., $|x/t-U_0^2|<\epsilon$ and $x/t>C$ for some $C>0$, the long-time asymptotic solution is dominated by the simple soliton as
			\begin{equation}
				u(x,t)=-\frac{U_0^2}{2} \operatorname{sech}^2\left(\frac{1}{2}(U_0x-U_0^3t+\log 2)\right)+O(t^{-l}).
			\end{equation}
			\item[(ii)] In the decay region (Region II), i.e., $|x/t-U_0^2|\geq\epsilon$ and $x/t>C$ for some $C>0$, one has
			\begin{equation}
				u(x,t)=O(t^{-l}).
			\end{equation}
			\item[(iii)] In self-similar region (Region III), i.e., $\tau=tk_0^3<C$ with $k_0=\sqrt{-\frac{x}{12t}}$ for some $C>0$, the long-time asymptotic solution is related with the Painlev\'e II transcendent by the form
			\begin{equation}\label{iii-self-similar}
				u(x,t)=\frac{1}{(3t)^{2/3}}\left( y^{2}\left( \frac{x}{(3t)^{1/3}}\right) +y'\left( \frac{x}{(3t)^{1/3}}\right) \right) +O(t^{-2/3}),
			\end{equation}
			where $y(s)$ with $s=\frac{x}{(3t)^{1/3}}$ is the Ablowitz-Segur solution \cite{Segur-Ablowitz-1981} to Painlev\'e II equation
			\begin{equation}\label{PII-E}
				y''(s)-sy(s)-2y^3(s)=0,
			\end{equation}
			with Stokes constants $(1,-1,0)$. Notice that the leading-order term in asymptotic solution (\ref{iii-self-similar}) is analogous to the Miura transformation (\ref{Miura-mKdV}), which is reasonable because the self-similar region in the long-time asymptotics of the mKdV equation (\ref{mKdV})  subjected to Schwartz class initial condition can be effectively described by the solution of Painlev\'e II equation.
			
			\item[(iv)] In the collisionless shock region (Region IV), i.e., $x<0,\ C^{-1}<\frac{-x}{(3t)^{1/3}(\log t)^{2//3}}<C$ for some $C>1$, the long-time asymptotic solution is approximated as
			\begin{equation}
				u(x,t)\sim -\frac{2x}{3t}(A(\alpha)+B(\alpha)\operatorname{cn}^2(2K(\alpha)\theta+\theta_0;\alpha)),
			\end{equation}
			where $\operatorname{cn}(\cdot;\alpha)$ is the Jacobi cnoidal function of modulus $\alpha$ and the other parameters are
			\[\alpha=1-\frac{a^2}{b^2},\ A=\frac{1}{4}(b^2-1),\ B=\frac{1}{2}(1-b^2),\ a^2+b^2=2,\]
			\[\theta=\frac{12\tau}{\pi}\int_{0}^{a}\sqrt{(p^2-a^2)(p^2-b^2)}\dd p,\]
			\[\theta_0=K(\alpha)-\int_{1}^{\sqrt{b/a}}((p^2-1)(1-(a/b)^2p^2))^{-1/2}\dd p-\frac{1}{2\pi}\int_{-1}^{1}\frac{\log(2\gamma a^2p^2)}{((p^2-1)((a/b)^2p^2)-1)^{1/2}}\dd p,\]
			and $a,b$ are modulated by the following implicit equation
			\[\frac{\log k_0^2}{\tau}=-24\int_{a}^{b}\sqrt{(p^2-a^2)(b^2-p^2)}\dd p.\]
			
			\item[(v)] In the dispersive wave region (Region V), i.e., $x/t\leq -C$ for some $C>0$, the long-time asymptotic solution is expressed by
			\begin{equation}				u(x,t)=\sqrt{\frac{4\nu(k_0)k_0}{3t}}\sin(16tk_0^3-\nu(k_0)\log(192tk_0^3)+\phi(k_0))+O(t^{-\alpha}),
			\end{equation}
			where $1/2<\alpha<1$ and
			\begin{equation}				\phi(k_0)=\frac{\pi}{4}-\arg(r(k_0))+\arg(\Gamma(\ii\nu(k_0)))-2\ii\chi(k_0)+4\arctan\left(\frac{U_0}{2k_0} \right).
			\end{equation}
		\end{itemize}
	\end{theorem}
	
	When the parameter $U_0<0$, it can be shown that the Schr\"{o}dinger operator $\mathcal{L}=\frac{\partial^2}{\partial x^2}-u$ with $u=-U_0\delta(x)$ does not support discrete spectrum and the soliton region is absent in this case. Thus the following theorem holds.
	
	\begin{theorem}\label{theorem-2} \rm
		For $U_0<0$, let $l>0$ be an integer, then as $t\to \infty$, the long-time asymptotic behaviors of the initial value problem $(\ref{InitialProblem})$ are formulated in the following forms (see Figure \ref{x0t-plane-U0<0})
		\begin{itemize}
			\item[(i)] In the decay region (Region I), $x/t>C$ for some $C>0$, one has
			\begin{equation}
				u(x,t)=O(t^{-l}).
			\end{equation}
			\item[(ii)] In self-similar region (Region II), i.e., $\tau=tk_0^3<C$ with $k_0=\sqrt{-\frac{x}{12t}}$ for some $C>0$, the long-time asymptotic solution is related with the Painlev\'e II transcendent by the form
			\begin{equation}
				u(x,t)=\frac{1}{(3t)^{2/3}}\left( y^{2}\left( \frac{x}{(3t)^{1/3}}\right) +y'\left( \frac{x}{(3t)^{1/3}}\right) \right) +O(t^{-2/3}),
			\end{equation}
			where $y(s)=y(s;1,-1,0)$ is the Ablowitz-Segur solution \cite{Segur-Ablowitz-1981} to the Painlv\'e II equation (\ref{PII-E}).
			\item[(iii)] In the collisionless shock region (Region III), i.e., $x<0,\ C^{-1}<\frac{-x}{(3t)^{1/3}(\log t)^{2//3}}<C$ for some $C>1$, the long-time asymptotic solution is approximated as
			\begin{equation}
				u(x,t)\sim -\frac{2x}{3t}(A(\alpha)+B(\alpha)\operatorname{cn}^2(2K(\alpha)\theta+\theta_0;\alpha)),
			\end{equation}
			where $\operatorname{cn}(\cdot;\alpha)$ is the Jacobi cnoidal function of modulus $\alpha$ the other parameters are
			\[\alpha=1-\frac{a^2}{b^2},\ A=\frac{1}{4}(b^2-1),\ B=\frac{1}{2}(1-b^2),\ a^2+b^2=2,\]
			\[\theta=\frac{12\tau}{\pi}\int_{0}^{a}\sqrt{(p^2-a^2)(p^2-b^2)}\dd p,\]
			\[\theta_0=K(\alpha)-\int_{1}^{\sqrt{b/a}}((p^2-1)(1-(a/b)^2p^2))^{-1/2}\dd p-\frac{1}{2\pi}\int_{-1}^{1}\frac{\log(2\gamma a^2p^2)}{((p^2-1)((a/b)^2p^2)-1)^{1/2}}\dd p,\]
			and $a,b$ are modulated by the following implicit equation
			\[\frac{\log k_0^2}{\tau}=-24\int_{a}^{b}\sqrt{(p^2-a^2)(b^2-p^2)}\dd p.\]
			\item[(iv)] In the dispersive wave region (Region IV), i.e., $x/t\leq -C$ for some $C>0$, the long-time asymptotic solution is expressed by
			\begin{equation}				u(x,t)=\sqrt{\frac{4\nu(k_0)k_0}{3t}}\sin(16tk_0^3-\nu(k_0)\log(192tk_0^3)+\phi(k_0))+O(t^{-\alpha}),
			\end{equation}
			where
			\begin{equation}
				\phi(k_0)=\frac{\pi}{4}-\arg(r(k_0))+\arg(\Gamma(\ii\nu(k_0)))-2\ii\chi(k_0).
			\end{equation}
		\end{itemize}
	\end{theorem}
	\begin{remark}
		We only present the asymptotic expressions of the collisionless shock region, i.e., the Region IV in Theorem \ref{theorem-1} and Region III in Theorem \ref{theorem-2} without proof. Moreover, so far, there have been no results related to the ``Region T" in Figure \ref{x0t-plane}.
	\end{remark}
	\par
	Figure \ref{U02T50v20} and Figure \ref{U02T50v20-2} show the comparisons of theoretical results in Theorem \ref{theorem-1} with the direct numerical simulations in different regions for parameter $U_0=2$ and time $t=50$. Figure \ref{U02T50v20} gives an overall comparison illustration, while Figure \ref{U02T50v20-2} shows the comparison diagrams in four different regions which are soliton region, self-similar region, collisionless shock region and dispersive wave region. It is observed that the results from asymptotic analysis are in very good agreement with the numerical simulations, which demonstrates the correctness of the theoretical calculations.
	\begin{figure}[H]
		\centering
		\includegraphics[scale=0.7]{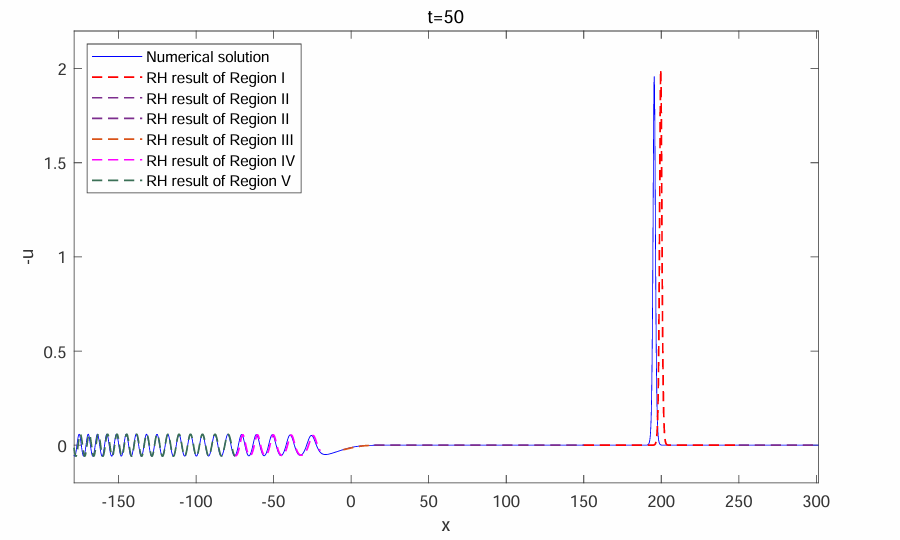}
		\caption{The comparisons of the leading-order terms in Theorem \ref{theorem-1} with the result of numerical simulation for $U_0=2$ and $t=50$, where the solid line represents the result of numerical simulation and the dashed lines represent the result of asymptotic analysis. }
		\label{U02T50v20}
	\end{figure}
	\begin{figure}[H]
		\centering
		\begin{subfigure}{0.45\textwidth}
			\includegraphics[scale=0.4]{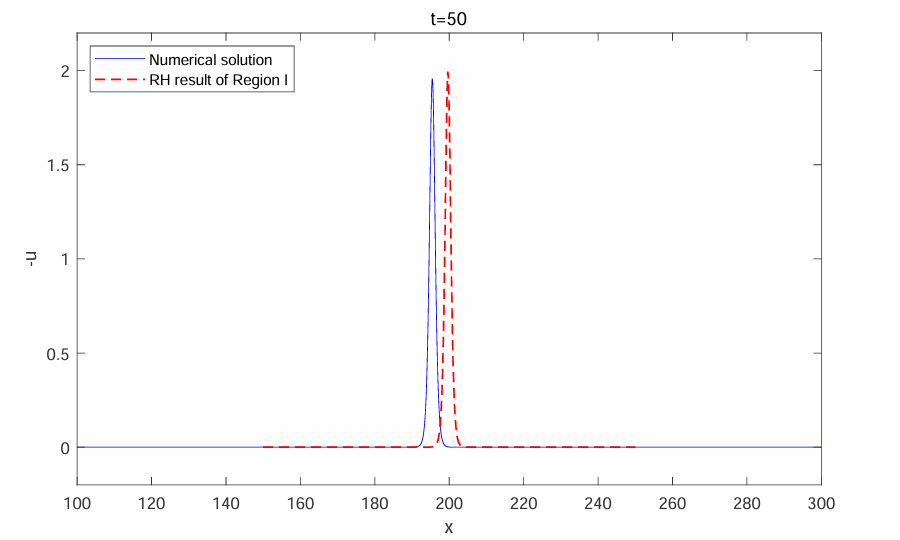}
			\caption{Soliton Region.}
			\label{U02T50Sl}
		\end{subfigure}
		\begin{subfigure}{0.45\textwidth}
			\includegraphics[scale=0.4]{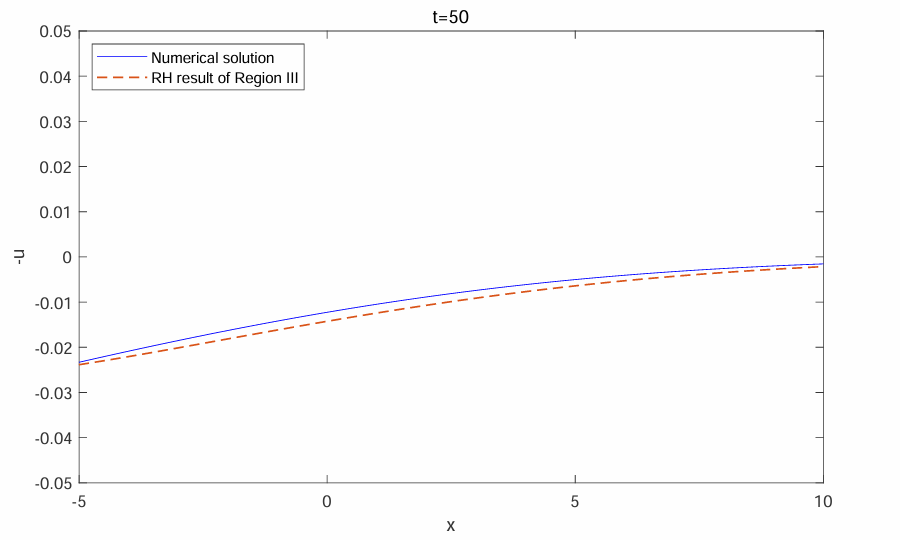}
			\caption{Self-similar Region.}
			\label{}
		\end{subfigure}
		\begin{subfigure}{0.45\textwidth}
			\includegraphics[scale=0.4]{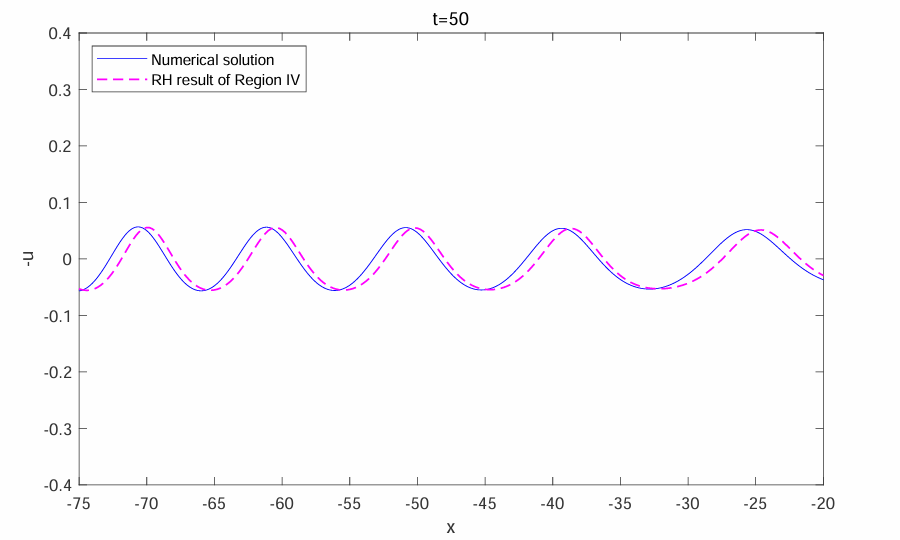}
			\caption{Collisionless shock Region.}
			\label{}
		\end{subfigure}
		\begin{subfigure}{0.45\textwidth}
			\includegraphics[scale=0.4]{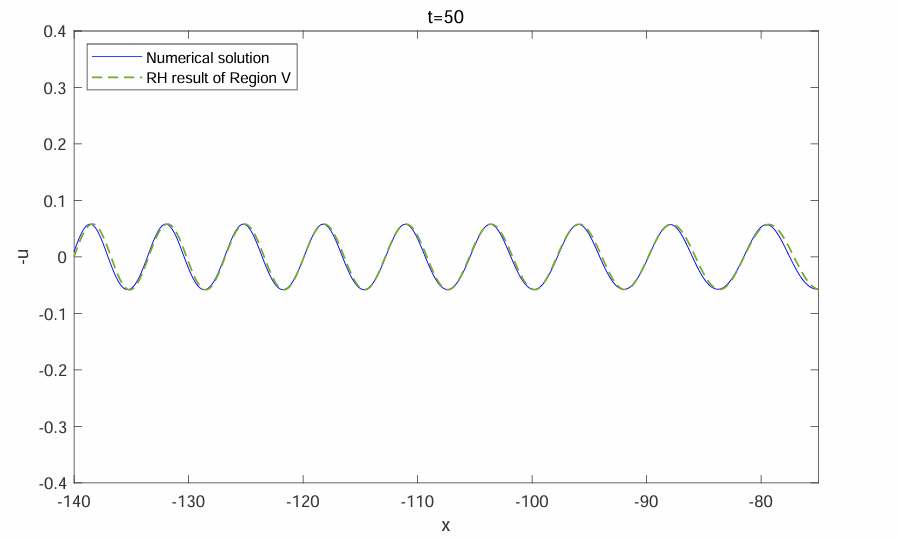}
			\caption{Dispersive wave Region.}
			\label{}
		\end{subfigure}
		\caption{The individual comparisons of the leading-order terms in Theorem \ref{theorem-1} with the result of the numerical simulation in four different regions which are soliton region, self-similar region, collisionless shock region and dispersive wave region, where the parameter is $U_0=2$ and the time is $t=50$. }
		\label{U02T50v20-2}
	\end{figure}
	
	\begin{remark}
		As depicted in Figure \ref{U02T50v20} and Figure \ref{U02T50Sl}, the numerical comparisons within the soliton region exhibit notable discrepancies. This divergence arises primarily from the inherent limitations of computational processing. Specifically, computers are confined to handling finite values, whereas the delta function $\delta(x)$ is characterized by its infinite value at the origin. Consequently, in numerical simulations, it is both practical and reasonable to approximate the delta function $\delta(x)$ using an extremely narrow rectangular potential barrier. This approximation effectively bridges the gap between theoretical constructs and computational feasibility, thereby enhancing the accuracy and reliability of the simulation results.
		
	\end{remark}
	
	\begin{figure}[H]
		\centering
		\includegraphics[scale=0.7]{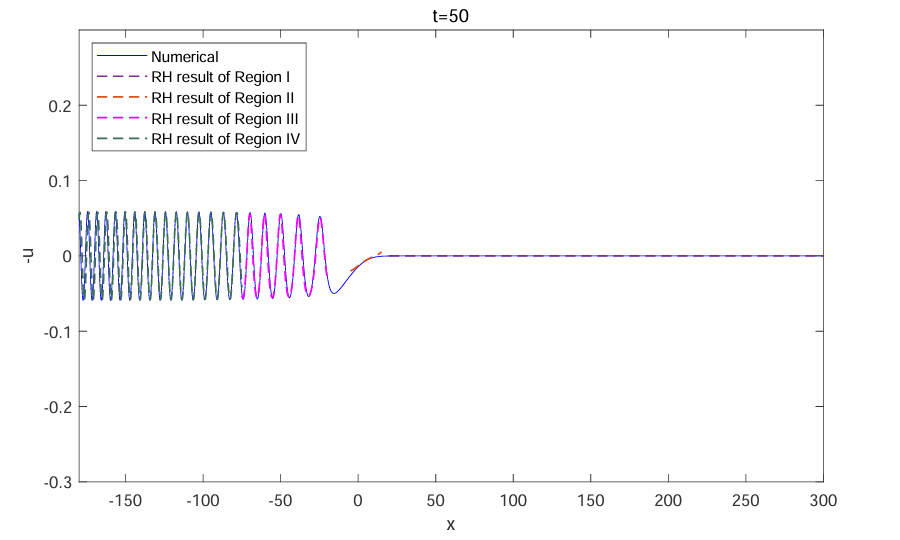}\label{}
		\caption{The comparisons of the leading-order terms in Theorem \ref{theorem-1} with the result of numerical simulation for $U_0=-2$ and $t=50$, where the solid line represents the result of numerical simulation and the dashed lines represent the result of asymptotic analysis.}
	\end{figure}
	\begin{figure}[H]
		\centering
		\begin{subfigure}{0.45\textwidth}
			\includegraphics[scale=0.4]{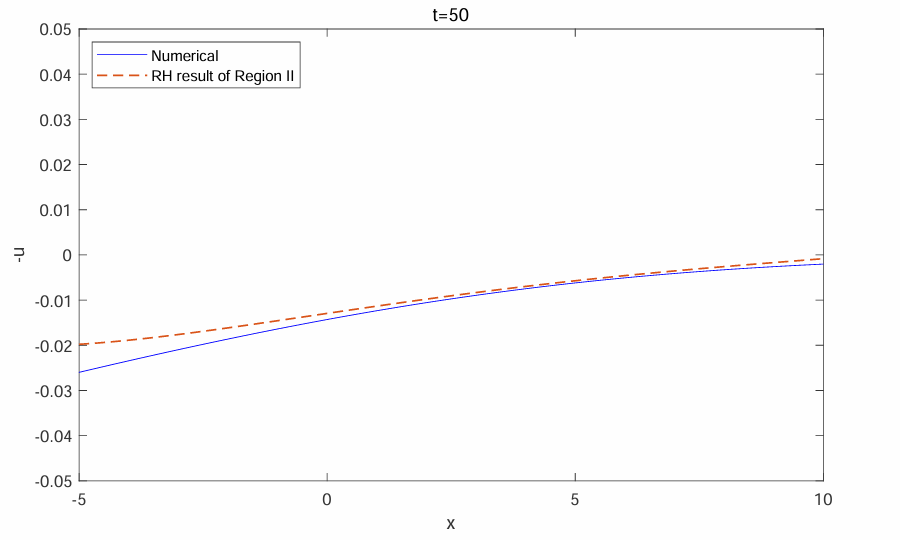}
			\caption{Self-similar Region.}
			\label{}
		\end{subfigure}
		\begin{subfigure}{0.45\textwidth}
			\includegraphics[scale=0.4]{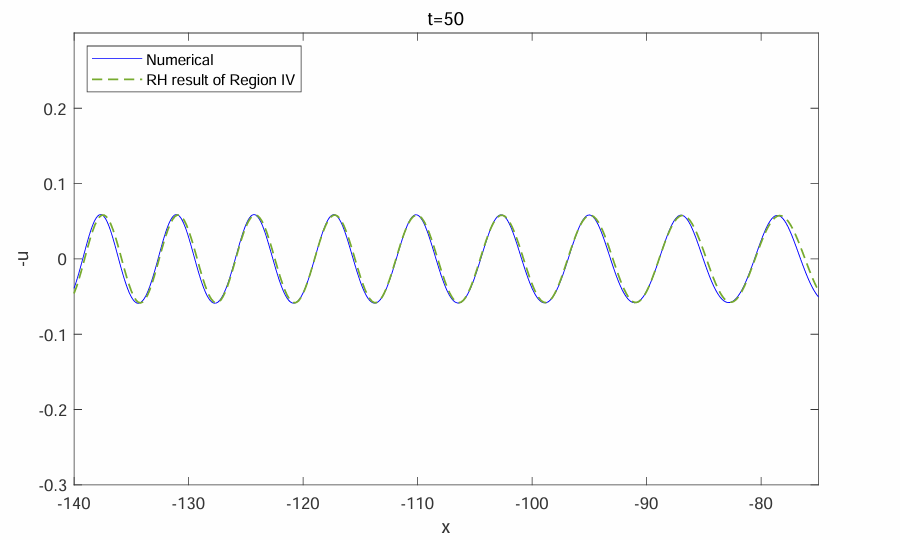}
			\caption{Dispersive wave Region.}
			\label{}
		\end{subfigure}
		\caption{The individual comparisons of the leading-order terms in Theorem \ref{theorem-1} with the result of the numerical simulation in four different regions which are dispersive wave region and self-similar region, where the parameter is $U_0=-2$ and the time is $t=50$.}
	\end{figure}

	\section{Inverse Scattering Transform and RH Problem}
	
	The KdV equation in initial problem (\ref{InitialProblem}) admits the linear spectrum problem, i.e., Lax pair
	\begin{equation}\label{Lax-pair}
		\begin{cases}
			\Psi_x=\tilde{L}\Psi,\\
			\Psi_t=\tilde{A}\Psi,
		\end{cases}
	\end{equation}
	where the matrices $\tilde{L}$ and $\tilde{A}$ are
	\begin{equation}
		\tilde{L}=\begin{pmatrix}
			0&1\\
			\lambda+u&0
		\end{pmatrix},\quad
		\tilde{A}=\begin{pmatrix}
			-u_x&2u-4\lambda\\
			-u_{xx}-(\lambda+u)(4\lambda-2u)&u_x
		\end{pmatrix}.
	\end{equation}
	The compatibility condition $\tilde{L}_t-\tilde{A}_x+[\tilde{L},\tilde{A}]=0$ yields the KdV equation.
	\par
	Let $\lambda=-k^2$ and take
	\begin{equation}
		J=\begin{pmatrix}
			\ii&-\ii\\
			k & k
		\end{pmatrix},
	\end{equation}
	then $\det(J)=2\ii k$. When $k\neq0$, let $\Psi=J\Phi$ then the Lax pair (\ref{Lax-pair}) becomes
	\begin{equation}
		\begin{cases}
			\Phi_x=L\Phi,\\
			\Phi_t=A\Phi,
		\end{cases}
		\label{lax-Phi}
	\end{equation}
	where $L:=J^{-1}\tilde{L}J$ and $A:=J^{-1}\tilde{A}J$. Direct calculations follow that
	\begin{equation}
		\begin{cases}
			L=-\ii k\sigma_3+k^{-1}L_{-1},\\
			A=-4\ii k^3\sigma_3+kA_1+A_0+k^{-1}A_{-1},
		\end{cases}
	\end{equation}
	where
	\begin{equation}
		L_{-1}=\frac{\ii u}{2}Q,\quad A_{1}=2u\sigma_2,\quad
		A_{0}=u_x\sigma_1,\quad
		A_{-1}=(\ii u^2-\frac{\ii}{2} u_{xx})Q,
	\end{equation}
	and
	\begin{equation}
		Q=\begin{pmatrix}
			1&-1\\
			1&-1
		\end{pmatrix},\quad
		\sigma_1=\begin{pmatrix}
			0&1\\
			1&0
		\end{pmatrix},\quad
		\sigma_2=\begin{pmatrix}
			0&-\ii\\
			\ii&0
		\end{pmatrix},\quad
		\sigma_3=\begin{pmatrix}
			1&0\\
			0&-1
		\end{pmatrix}.
	\end{equation}
	\par
	Define the Jost function
	\begin{equation}
		N(x,t,k)=\Phi(x,t,k) \ee^{(\ii kx+4\ii k^3t)\sigma_3},
	\end{equation}
	then $N(x,t,k)\to I$ as $|x|\to \infty$, and $N$ satisfies the new linear spectrum problem
	\begin{equation}
		\begin{cases}
			N_x+\ii k[\sigma_3,N]=k^{-1}L_{-1}N=:UN,\\
			N_t+4\ii k^3[\sigma_3,N]=(kA_1+A_0+k^{-1}A_{-1})N=:VN,
		\end{cases}
	\end{equation}
	which indicates that
	\begin{equation}
		\dd (\ee^{(\ii kx+4\ii k^3t)\hat{\sigma}_3}N(x,t,k))={\ee}^{(\ii kx+4\ii k^3t)\hat{\sigma}_3}((U\dd x+V\dd t)N(x,t,k)).
	\end{equation}
	Integrating this total derivative along $(-\infty,x)$ and $(x,\infty)$ and using the asymptotics of the function $N(x,t,k)$, yield the Volterra integral equations
	\begin{equation}
		\begin{aligned}
			N_l(x,t,k)&=I+\int_{-\infty}^{x}{\ee}^{-\ii k(x-s)\hat{\sigma}_3}(U(s,t,k)N_l(s,t,k))\dd s,\\
			N_r(x,t,k)&=I-\int_{x}^{\infty}{\ee}^{-\ii k(x-s)\hat{\sigma}_3}(U(s,t,k)N_r(s,t,k))\dd s,
		\end{aligned}
		\label{volterra-integral-equations}
	\end{equation}
	which can be solved in the following proposition.
	
	\subsection{The Jost function for $u(x,0)=-U_0\delta(x)$}
	
	In this subsection, the Jost functions $N_l$ and $N_r$ are solved with initial data $u(x,0)=-U_0\delta(x)$ by using the Volterra integral equations (\ref{volterra-integral-equations}). The more general case of $u(x,0)=-\sum_{n=1}^LU_n\delta_n(x)$ in (\ref{L-delta-intial}) will be considered in Section \ref{section-general}.
	\begin{Proposition}
		For the initial condition $u(x,0)=-U_0\delta(x)$, the functions $N_l(x,0,k)$ and $N_r(x,0,k)$ in (\ref{volterra-integral-equations}) can be solved by
		\begin{equation}\label{volterra-solution-1}
			N_l(x,0,k)=\begin{cases}
				I,&x\leq 0,\\
				I-\frac{\ii U_0}{2k}{\ee}^{-\ii kx\hat{\sigma}_3}Q,&x>0.
			\end{cases}
		\end{equation}
		\begin{equation}\label{volterra-solution-2}
			N_r(x,0,k)=\begin{cases}
				I+\frac{\ii U_0}{2k}{\ee}^{-\ii kx\hat{\sigma}_3}Q,&x<0,\\
				I,&x\geq0.
			\end{cases}
		\end{equation}
		\label{t=0-jost-function}
	\end{Proposition}
	\begin{proof}
		When $x=0$, integrating by parts we have
		\begin{equation}
			\begin{aligned}
				N_l(0,0,k)&=I+\int_{-\infty}^{0}{\ee}^{\ii ks\hat{\sigma}_3}(U(s,0,k)N_l(s,0,k))\dd s\\
				&=I-\frac{\ii U_0}{2k}\int_{-\infty}^{0}\delta(s){\ee}^{\ii ks\hat{\sigma}_3}(QN_l(s,0,k))\dd s\\
				&=I-\frac{\ii U_0}{2k}\left. H(s){\ee}^{\ii ks\hat{\sigma}_3}(QN_l(s,0,k))\right|_{s=-\infty}^{s=0}+\frac{\ii U_0}{2k}\int_{-\infty}^{0}H(s)\dd({\ee}^{\ii ks\hat{\sigma}_3}(QN_l(s,0,k)))\\
				&=I,
			\end{aligned}
		\end{equation}
		where $H(s)$ is Heaviside function satisfying $H(s)=0,x\leq0,H(s)=1,x>0$, and $H^{\prime}(s)=\delta(s)$.
		\par	
		When $x>0$, reminding the properties of the delta function $\delta(x)$, it is found that
		\begin{equation}
			\begin{aligned}
				N_l(x,0,k)&=I+\int_{-\infty}^{x}{\ee}^{-\ii k(x-s)\hat{\sigma}_3}(U(s,0,k)N_l(s,0,k))\dd s\\
				&=I+\int_{-x}^{x}{\ee}^{-\ii k(x-s)\hat{\sigma}_3}(U(s,0,k)N_r(s,0,k))\dd s\\
				&=I-\frac{\ii U_0}{2k}\int_{-x}^{x}\delta(s){\ee}^{-\ii k(x-s)\hat{\sigma}_3}(QN_l(s,0,k))\dd s\\
				&=I-\left.\frac{\ii U_0}{2k}{\ee}^{-\ii k(x-s)\hat{\sigma}_3}(QN_l(s,0,k))\right|_{s=0}\\
				&=I-\frac{\ii U_0}{2k}{\ee}^{-\ii kx\hat{\sigma}_3}(QN_l(0,0,k))\\
				&=I-\frac{\ii U_0}{2k}{\ee}^{-\ii kx\hat{\sigma}_3}Q.
			\end{aligned}
		\end{equation}
		\par	
		When $x<0$, it follows that
		\begin{equation}
			\begin{aligned}
				N_l(x,0,k)&=I+\int_{-\infty}^{x}{\ee}^{-\ii k(x-s)\hat{\sigma}_3}(U(s,0,k)N_l(s,0,k))\dd s\\
				&=I-\frac{\ii U_0}{2k}\int_{-\infty}^{x}\delta(s){\ee}^{-\ii k(x-s)\hat{\sigma}_3}(QN_l(s,0,k))\dd s\\
				&=I.
			\end{aligned}
		\end{equation}
		\par
		In this similar way, the solution $N_r(x,0,k)$ in (\ref{volterra-solution-2}) of the Volterra integral equation can also be derived.
	\end{proof}
	\par	
	Since $\Phi_l(x,t,k)=N_l(x,t,k)\ee^{-(\ii kx+4\ii k^3t)\sigma_3}$ and $\Phi_r(x,t,k)=N_r(x,t,k)\ee^{-(\ii kx+4\ii k^3t)\sigma_3}$ are two linearly dependent solutions of the Lax pair $(\ref{lax-Phi})$, there exists a matrix $S(t,k)=(s_{ij}(t,k))_{2\times2}$ for $t\geq 0$ and $k\in\C\backslash\{0\}$, which is independent of $x$, such that
	\begin{equation}
		\Phi_l(x,t,k)=\Phi_r(x,t,k)S(t,k),
	\end{equation}
	where $S(t,k)$ is expressed by
	\begin{equation}
		S(t,k)={\ee}^{\ii (kx+4k^3t)\hat{\sigma}_3}(N_r^{-1}(x,t,k)N_l(x,t,k)).
		\label{scattering-matrax}
	\end{equation}
	\par
	Define the reflection coefficient
	\[r(t,k)=\frac{s_{21}(t,k)}{s_{11}(t,k)},\]
	then the scattering matrix $S(t,k)$ and the reflection coefficient $r(t,k)$ for $t=0$ can be obtained in the following proposition explicitly.
	
	\begin{Proposition}
		[Scattering data]
		For the initial condition $u(x,0)=-U_0\delta(x)$, it is immediate that
		\begin{equation}
			S(0,k)=I-\frac{\ii U_0}{2k}Q,\quad k\in\C\backslash\{0\},
		\end{equation}
		\begin{equation}
			r(0,k)=\frac{\ii U_0}{\ii U_0-2k} ,\quad k\in\C\backslash\{k_p\},
			\label{r(k)}
		\end{equation}
		where the pole of the reflection coefficient $r(0,k)$ is $k_p=\frac{\ii U_0}{2}.$
	\end{Proposition}
	\begin{proof} This proposition can be proved by using the definition of scattering matrix $S(k)$ in (\ref{scattering-matrax}) and the Proposition $\ref{t=0-jost-function}$ for $t=0$.
	\end{proof}
	\par
	For simplicity, denote $S(k)=S(0,k)$ and $r(k)=r(0,k).$
	
	\begin{Proposition}[Time evolution of the scattering data]
		For the initial condition $u(x,0)=-U_0\delta(x)$, the time evolutions of scattering data are given by
		\begin{equation}
			S(t,k)={\ee}^{-4\ii k^3t\hat{\sigma}_3}S(k),
		\end{equation}
		\begin{equation}
			r(t,k)={\ee}^{8\ii k^3t}r(k).
		\end{equation}
	\end{Proposition}

	\subsection{Construct the Riemann-Hilbert problem}
	\par
	In this subsection, the Jost functions are used to construct Riemann-Hilbert problem.
	First of all, the asymptotic behaviors of the Jost function $N(x,t,k)$ for $k\to\infty$ and $k\to0$ can be given in the proposition below.
	
	\begin{Proposition} The Jost function $N(x,t,k)$ admits the following asymptotic expansions for $k\to\infty$ and $k\to0$, respectively, i.e.,
		\par
		(1) (Asymptotic behavior as $k\to\infty$) $N(x,t,k)\to I+O(\frac{1}{k})$ as $\ k\to\infty$.
		\par
		(2) (Asymptotic behavior as $k\to0$) $N(x,t,k)\to\alpha(x)Qk^{-1}+O(1)$ as $\ k\to0$, where $\alpha(x)$ is a real valued function.
	\end{Proposition}
	\par
	The solution of the initial problem (\ref{InitialProblem}) for KdV equation is reconstructed through the formula
	\begin{equation}\label{reconstructed-formula}
		u(x,t)=-2\ii\frac{\partial}{\partial x}\lim_{k\to\infty}k(N(x,t,k)-I)_{11}.
	\end{equation}
	
	\begin{Proposition}[Analyticity] Assume $u(\cdot,t)\in\L^1(\R)$ for all $t>0$ and let $N_j=(N_{j1},N_{j2})$ for $j\in\{l,r\}$, then the Jost function possesses the following properties:
		\par
		(1) $N_{l1}(x,t,k)$ is analytic for $k\in\C_{+}$, $N_{l2}(x,t,k)$ is analytic for $k\in\C_{-}$, and both of them are continuous up to $k\in \R\backslash\{0\}$.
		\par
		(2) $N_{r1}(x,t,k)$ is analytic for $k\in\C_{-}$, $N_{r2}(x,t,k)$ is analytic for $k\in\C_{+}$, and both of them are continuous up to $k\in \R\backslash\{0\}$.
		\par
		(3) $s_{11}(k)$ is analytic for $k\in\C_{+}$, $s_{22}(k)$ is analytic for $k\in\C_{-}$, while $s_{12}(k)$ and $s_{21}(k)$ are continuous up to $k\in \R$.
		\label{analyticity}
	\end{Proposition}
	\begin{proof}
		1. In what follows, we only prove the analyticity of $N_{l1}(x,t,k)$, and the proofs of the analyticity for other Jost functions are similar. Equation $(\ref{volterra-integral-equations})$ indicates that
		\begin{equation}
			N_{l1}(x,t,k)=I+\int_{-\infty}^{x}\operatorname{diag}(1,\ee^{2\ii k(x-s)})U(s,t,k)N_{l1}(s,t,k)\dd s.
			\label{N_l1-intergral-equation}
		\end{equation}
		\par
		Construct a sequence of functions in the forms
		\begin{equation}
			\begin{aligned}
				&N^{(0)}_{l1}=(1,0)^T,\\
				&N^{(j+1)}_{l1}(x,t,k)=\int_{-\infty}^{x}\operatorname{diag}(1,\ee^{2\ii k(x-s)})U(s,t,k)N^{(j)}_{l1}(s,t,k)\dd s,\quad j\geq0.
			\end{aligned}
			\label{N_l1-recursion-formula}
		\end{equation}
		For any $\epsilon>0$, letting $\C_{+}^{\epsilon}=\C\backslash D(0,\epsilon)=\C\backslash\{k\in\C: |k|<\epsilon\}$ and denoting
		\[\rho(x,t):=\frac{4}{\epsilon}\int_{-\infty}^{x}|u(s,t)|\dd s,\]
		by induction on $\C_{+}^{\epsilon}$, it can be obtained that
		\begin{equation}
			||N^{(j)}_{l1}(x,t,k)||\leq \frac{(\rho(x,t))^j}{j!},\quad j=0,1,2,\cdots,
		\end{equation}
		where $||M||$ represents the sum of the absolute values of each element of the matrix $M$. Consequently, the integrals defined by $(\ref{N_l1-recursion-formula})$ converge and are bounded. From the analyticity of $N^{(0)}_{l1}$, it is known that $N^{(j)}_{l1}(x,t,k)$ is analytic for $k\in\C_{+}$.
		\par
		Fix $a\in\R$, then $\rho(x,t)\leq\rho(a,t)$ for $x\leq a$. Thus, for $x\leq a$ and $k\in\C^{\epsilon}_{+}$, one has
		\[||N^{(j)}_{l1}(x,t,k)||\leq \frac{(\rho(a,t))^j}{j!},\quad j=0,1,2,\cdots.\]
		\par
		Therefore, the series $\sum_{j=0}^{\infty}N_{l1}(x,t,k)$ converges absolutely and uniformly for $x\leq a$ and $k\in\C^{\epsilon}_{+}$, and it is analytic on $\C^{\epsilon}_{+}$.
		\par
		From the recursion formula and the uniform convergence above, it is obvious that $\sum_{j=0}^{\infty}N_{l1}(x,t,k)$ satisfies the integral equation in $(\ref{N_l1-intergral-equation})$. Finally, due to the arbitrariness of $\epsilon$, it follows that the Jost function $N_{l1}(x,t,k)$ is analytic for $\operatorname{Im}k>0$.
		\par
		2. Now we prove that $N_{l1}(x,t,k)$ is continuous up to $\R\backslash\{0\}$ for $k\in\overline{{\C}_{+}}$. Assuming $\epsilon\in(0,k_0/2)$ and denoting $K(s,t,k):=\operatorname{diag}(1,\ee^{2\ii k(x-s)})U(s,t,k)$, then for $k_0\in\R\backslash\{0\}, \operatorname{Im}k>0$, we have
		\begin{equation}
			\begin{aligned}
				||N_{l1}(x,t,k)-N_{l1}(x,t,k_0)||\leq& \int_{-\infty}^{x}||K(s,t,k)||||N_{l1}(s,t,k)-N_{l1}(s,t,k_0)||\dd s\\
				&+\int_{-\infty}^{x}||N_{l1}(s,t,k_0)||||K(s,t,k)-K(s,t,k_0)||\dd s.
			\end{aligned}
			\label{continuity-estimate1}
		\end{equation}

		Since $\operatorname{Im}k>0$, it can be seen that
		\begin{equation}
			||K(s,t,k)||\leq2||U(s,t,k)||\leq2\left\| \frac{\ii u(s,t)}{2k}Q\right\| \leq \frac{4}{|k|}|u(s,t)|,
			\label{continuity-estimate2}
		\end{equation}
		so one has
		\begin{equation}
			\begin{aligned}
				||N_{l1}(x,t,k)||&\leq1+\int_{-\infty}^{x}||K(s,t,k)||||N_{l1}(s,t,k)||\dd s\\
				&\leq1+\frac{4}{|k|}\int_{-\infty}^{x}|u(s,t)|||N_{l1}(s,t,k)||\dd s.
			\end{aligned}
		\end{equation}
		\par
		The Grownwall's inequality shows that
		\begin{equation}
			\begin{aligned}
				||N_{l1}(x,t,k)||&\leq1+\int_{-\infty}^{x}\frac{4}{|k|}|u(s,t)|\exp\left(\int_{s}^{x}\frac{4}{|k|}|u(r,t)|\dd r \right)\dd s \\
				&\leq1+\frac{4}{|k|}\int_{\R}|u(s,t)|\dd s\exp\left(\frac{4}{|k|}\int_{\R}|u(s,t)|\dd s\right) =:1+c\ee^c<\infty,
			\end{aligned}
			\label{continuity-estimate3}
		\end{equation}
		where $c=\frac{4}{|k|}\int_{\R}|u(s,t)|\dd s$.
		
		From Eqs. $(\ref{continuity-estimate2})$ and $(\ref{continuity-estimate3})$, it is known that
		\begin{equation}
			||N_{l1}(\cdot,t,k_0)||||K(\cdot,t,k)-K(\cdot,t,k_0)||\in\L^1(\R).
		\end{equation}
		By the Lebesgue dominated convergence theorem, there exists a constant $\eta_1=\eta_1(\epsilon)>0$ such that
		\begin{equation}
			\int_{-\infty}^{x}||N_{l1}(s,t,k_0)||||K(s,t,k)-K(s,t,k_0)||\dd s<\epsilon,
			\label{continuity-estimate4}
		\end{equation}
		for $|k-k_0|<\eta_1$.
		\par
		Substituting the equations (\ref{continuity-estimate2}), (\ref{continuity-estimate3}) and (\ref{continuity-estimate4}) into (\ref{continuity-estimate1}) yields
		\begin{equation}
			||N_{l1}(x,t,k)-N_{l1}(x,t,k_0)||\leq \frac{4}{|k|}\int_{-\infty}^{x}|u(s,t)|||N_{l1}(s,t,k)-N_{l1}(s,t,k_0)||\dd s+\epsilon.
		\end{equation}
		\par
		Thus the Grownwall's inequality again indicates that
		\begin{equation}
			||N_{l1}(x,t,k)-N_{l1}(x,t,k_0)||\leq\epsilon c,
		\end{equation}
		further	taking $\eta=\min\{\epsilon,\eta_1\}$, then for $|k-k_0|<\eta$, one obtains
		\begin{equation}
			||N_{l1}(x,t,k)-N_{l1}(x,t,k_0)||\leq\epsilon\frac{2}{|k_0|}||u(\cdot,t)||_{\L^1(\R)}.
		\end{equation}
		Finally, due to the arbitrariness of $\epsilon$, it follows that $N_{l1}(x,t,k)$ is continuous with respect to $\R\backslash\{0\}$ in $k\in\overline{{\C}_{+}}$.
	\end{proof}
	\begin{Proposition}[Symmetry] The Jost function and scattering matrix satisfy the following symmetries:
		$$
		N(x,t,k)=\sigma_1\overline{N(x,t,\bar{k})}\sigma_1=\sigma_1N(x,t,-k)\sigma_1,
		~S(t,k)=\sigma_1\overline{S(t,\bar{k})}\sigma_1=\sigma_1S(t,-k)\sigma_1.
		$$
	\end{Proposition}
	\par
	\par
	Now it is ready to construct the Riemann-Hilbert problem associated with the initial problem (\ref{InitialProblem}) of KdV equation.
	\par
	According to Proposition $\ref{analyticity}$, define a sectionally analytic matrix-valued function
	\begin{equation}
		M(x,t,k)=\begin{cases}
			\begin{pmatrix}
				\frac{N_{l1}(x,t,k)}{s_{11}(k)}&N_{r2}(x,t,k)
			\end{pmatrix},\ \operatorname{Im}k>0,\\
			\begin{pmatrix}
				N_{r1}(x,t,k)&\frac{N_{l2}(x,t,k)}{s_{22}(k)}
			\end{pmatrix},\ \operatorname{Im}k<0.
		\end{cases}
	\end{equation}
	It should be noted that for $U_0>0$, the function $M(x,t,k)$ has two poles $k_p,-k_p$ on $\C\backslash\R$, while for $U_0<0$, $M(x,t,k)$ does not have any pole on $\C\backslash\R$. Without loss of generality, we only give the RH problem associated with the initial problem (\ref{InitialProblem}) for $U_0>0$.
	
	Denote the phase function
	\begin{equation}
		\label{theta}
		\theta(x,t,k):=k\frac{x}{t}+4k^3.
	\end{equation}
	\begin{RHP}\label{RHP-3.1}
		The function $M(x,t,k)$ satisfies the following properties:
		\par
		(1) $M(x,t,\cdot):\C\backslash\R\to \C^{2\times2}$ is analytic;
		\par
		(2) $M(x,t,k)=\sigma_1\overline{M(x,t,\bar{k})}\sigma_1=\sigma_1M(x,t,-k)\sigma_1$;
		\par
		(3) $M(x,t,k)\to I,\  k\to\infty,k\notin\R$;
		\par
		(4) $M(x,t,k)\to\alpha(x)Qk^{-1},\ k\to0,\ k\notin\R$, where $\alpha(x)$ is a real-valued function;
		\par
		(5) $M_{+}(x,t,k)=M_{-}(x,t,k)V(x,t,k)$ for $k\in\R\backslash\{0\}$, where the jump matrix is
		\begin{equation}
			V(x,t,k)=\begin{pmatrix}
				1-|r(k)|^2&-\overline{r(k)}{\ee}^{-2\ii t\theta(x,t)}\\
				r(k){\ee}^{2\ii t\theta(x,t)}&1
			\end{pmatrix};
		\end{equation}
		\par
		(6) For $U_0>0$, $M(x,t,k)$ satisfies the residue conditions at $k_p,-k_p$
		\begin{equation}
			\Res_{k=k_p}M(x,t,k)=\lim_{k\to k_p}M(x,t,k)
			\begin{pmatrix}
				0&0\\
				\gamma_{p}\ee^{2\ii t\theta(k_p)}&0
			\end{pmatrix},
			\label{res_M_at_kp}
		\end{equation}
		\begin{equation}
			\Res_{k=-k_p}M(x,t,k)=\lim_{k\to -k_p}M(x,t,k)
			\begin{pmatrix}
				0&-\gamma_{p}\ee^{2\ii t\theta(k_p)}\\
				0&0
			\end{pmatrix},
			\label{res_M_at_-kp}
		\end{equation}
		where $\gamma_{p}:=\frac{s_{21}(k_p)}{s^{\prime}_{11}(k_p)}$. In fact, it can be known that $\gamma_{p}=-k_{p}=-\frac{\ii U_0}{2}$.
	\end{RHP}
	To remove the first-order pole of $M(x,t,k)$ at $k=0$ (see item (4) in the RH problem \ref{RHP-3.1}), define $\tilde{M}(x,t,k)$ as
	\begin{equation}
		\tilde{M}(x,t,k)=\left(I-\frac{\ii y(x,t)}{k}Q \right)M(x,t,k) ,
	\end{equation}
	where
	\begin{equation}
		y(x,t)=-\ii \lim_{k\to\infty}k\tilde{M}_{21}(x,t,k).
	\end{equation}
	
	Then the matrix-valued function $\tilde{M}(x,t,k)$ satisfies the following Riemann-Hilbert problem:
	
	\begin{RHP} The function
		$\tilde{M}(x,t,k)$ satisfies the following properties:
		\par
		(1) $\tilde{M}(x,t,\cdot): ~\C\backslash\R\to \C^{2\times2}$ is analytic;
		\par
		(2) $\tilde{M}(x,t,k)=\sigma_1\overline{\tilde{M}(x,t,\bar{k})}\sigma_1=\sigma_1\tilde{M}(x,t,-k)\sigma_1$;
		\par
		(3) $\tilde{M}(x,t,k)\to I,\quad k\to\infty,k\notin\R$;
		\par
		(4) $\tilde{M}_{+}(x,t,k)=\tilde{M}_{-}(x,t,k)\tilde{V}(x,t,k)$ for $k\in\R\backslash\{0\}$, where the jump matrix is
		\begin{equation}
			\tilde{V}(x,t,k)=\begin{pmatrix}
				1-|r(k)|^2&-\overline{r(k)}{\ee}^{-2\ii t\theta(x,t)}\\
				r(k){\ee}^{2\ii t\theta(x,t)}&1
			\end{pmatrix};
		\end{equation}
		\par
		(5) For $U_0>0$, $\tilde{M}(x,t,k)$ satisfies the residue conditions ($\ref{res_M_at_kp}$) and ($\ref{res_M_at_-kp}$) at $k_p,-k_p.$
		\label{RHP:tildeM}
	\end{RHP}
	
	\begin{theorem}
		The solution of the KdV equation is given by
		\begin{equation}\label{u=F(tildeM)}
			u(x,t)=-2\ii\frac{\partial}{\partial x}\lim_{k\to\infty}(k(\tilde{M}_{11}(x,t,k)-1))-2\ii\frac{\partial}{\partial x}\lim_{k\to\infty}k\tilde{M}_{21}(x,t,k).
		\end{equation}
	\end{theorem}
	\begin{proof}
		Denoting $\partial_x=\frac{\partial}{\partial x}$ and reminding the property $\tilde{M}(x,t,k)\to I$ as $k\to\infty$, it follows that
		\[\begin{aligned}
			u(x,t)&=-2\ii\partial_x\lim_{k\to\infty}\left[k\left(\left( 1+\frac{\ii y}{k}\right) \tilde{M}_{11}-\frac{\ii y}{k}\tilde{M}_{21}-1\right) \right] \\
			&=-2\ii\partial_x\lim_{k\to\infty}\left[(k+\ii y)\tilde{M}_{11}-\ii y\tilde{M}_{21}-k\right]\\
			&=-2\ii\partial_x\lim_{k\to\infty}\left[k(\tilde{M}_{11}-1)+\ii y\tilde{M}_{11}-\ii y\tilde{M}_{21}\right]\\
			&=-2\ii\partial_x\left[\lim_{k\to\infty}k(\tilde{M}_{11}-1)+\lim_{k\to\infty}(k\tilde{M}_{21})\right].
		\end{aligned}\]
	\end{proof}

	\section{Proofs of the Main Results}
	\label{section-proofs}
	In this section, utilizing the reconstruction formula (\ref{u=F(tildeM)}), the long-time asymptotic behaviors of the solution to the initial problem (\ref{InitialProblem}) of KdV equation in Theorem \ref{theorem-1} and Theorem \ref{theorem-2} are analyzed by deforming the RH problem \ref{RHP:tildeM} through Deift-Zhou nonlinear steepest descent method \cite{Deift-Zhou-1993}.

	\subsection{\bf Case I. The parameter $U_0>0$}
	
	It will be seen that for parameter $U_0>0$ there are five regions in the initial problem (\ref{InitialProblem}) of KdV equation, which are soliton region, decay region, self-similar region, collisionless shock region and dispersive wave region. In what follow, we analyze these regions in detail except for the collisionless shock region.

	\subsubsection{The soliton region and decay region}
	
	When $x/t\geq C$ for $C>0$, the pole $k_p$ lies on the imaginary axis. Due to the relative positions of the pole $k_p$ in the signature of the ${\rm Re}(\ii\theta)$, different scenarios need to be discussed.
	\par
	In order to convert the residue condition into the jump condition, take the transformation for small $\epsilon>0$ of the form
	\begin{equation}
		M^{(0)}(x,t,k)=\begin{cases}
			\tilde{M}(x,t,k)\begin{pmatrix}
				1&0\\
				-\dfrac{\gamma_{p}\ee^{2\ii t\theta(k_p)}}{k-k_p}&1
			\end{pmatrix},&\quad k\in U_+,\\\\
			\tilde{M}(x,t,k)\begin{pmatrix}
				1&\dfrac{\gamma_{p}\ee^{2\ii t\theta(k_p)}}{k+k_p}\\
				0&1
			\end{pmatrix},&\quad k\in U_-,\\
			\tilde{M}(x,t,k),\quad \text{else},
		\end{cases}
	\end{equation}
	where $U_+=D(k_p,\epsilon)$ and $U_-=D(-k_p, \epsilon)$ are two disks with radius $\epsilon$. Further, define $\Gamma_{\pm}=\partial U_{\pm}$.
	
	Then the function $M^{(0)}(x,t,k)$ satisfies the following RH problem:
	\begin{RHP}
		Find a $2\times2$ matrix-valued function $M^{(0)}(x,t,k)$ with the following properties:\\
		(1) $M^{(0)}(x,t,\cdot)$ is sectionally analytic in $\C\backslash(\R\cup\Gamma_+\cup\Gamma_-)$;\\
		(2) $M^{(0)}(x,t,k)=\sigma_1M^{(0)}(x,t,-k)\sigma_1$;\\
		(3) $M^{(0)}(x,t,k)\to I,\quad k\to\infty,k\notin\R$;\\
		(4) $M^{(0)}_{+}(x,t,k)=M^{(0)}_{-}(x,t,k)V^{(0)}(x,t,k)$,
		where
		\begin{equation}
			V^{(0)}=
			\begin{cases}
				\begin{pmatrix}
					1&0\\
					-\dfrac{\gamma_p\ee^{2\ii t\theta(k_p)}}{k-k_p}&1
				\end{pmatrix},&\quad k\in\Gamma_+,\\\\
				\begin{pmatrix}
					1&-\dfrac{\gamma_p\ee^{2\ii t\theta(k_p)}}{k+k_p}\\
					0&1
				\end{pmatrix},&\quad k\in\Gamma_-,\\
				\tilde{V}(x,t,k),&\quad k\in\R.
			\end{cases}
		\end{equation}
		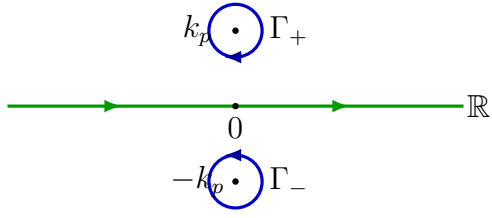
\begin{figure}[H]
			\centering
			\begin{tikzpicture}
				\draw[very thick,black!40!green,-latex](-3,0)--(-1.5,0);
				\draw[very thick,black!40!green,-latex](-1.7,0)--(1.5,0);
				\draw[very thick,black!40!green,-](1.3,0)--(3,0);
				
				\filldraw [black] (0,0) circle (1pt)node[below]{$0$};
				\filldraw [black] (3.2,0) circle (0pt)node[]{$\R$};
				
				\draw[black!20!blue, very thick] (0,1) circle (0.35);
				\filldraw [black] (0,1) circle (1pt);
				\draw[black!20!blue, very thick] (0,-1) circle (0.35);
				\filldraw [black] (0,-1) circle (1pt);
				
				\node[]at(0.7,1){$\Gamma_+$};
				\node[]at(0.7,-1){$\Gamma_-$};
				\node[]at(-0.5,1){$k_p$};
				\node[]at(-0.5,-1){$-k_p$};
				
				\draw[very thick,black!40!blue,-latex](0,0.65)--(-0.15,0.65);
				\draw[very thick,black!40!blue,-latex](0,-0.65)--(-0.15,-0.65);
			\end{tikzpicture}
			\caption{The contour for the RH problem with solution $M^{(0)}(x,t,k)$.}
		\end{figure}
		\label{RHP:pole_to_circle}
	\end{RHP}
	
	Letting $\theta^{\prime}(k)=0$, two stationary phase points $\pm k_0$ are obtained, where
	\begin{equation}
		k_0=\sqrt{-\frac{x}{12t}},
	\end{equation}
	in which $k_0\in\R$	for $x<0$ and $k_0\in\ii\R$ for $x>0$. Since
	\begin{equation} \operatorname{Re}(\ii\theta(k))=\operatorname{Im}k(4(\operatorname{Im}k)^2-12(\operatorname{Re}k)^2-\frac{x}{t}),
	\end{equation}
	the sign of $\operatorname{Re}(\ii\theta(k))$ is different for $x>0$ and $x<0$. For $x>0$, denote
	\begin{equation}
		\kappa_0=\sqrt{\frac{x}{4t}}.
	\end{equation}
	
	\begin{figure}[H]
		\centering
		\begin{subfigure}{0.45\textwidth}
			\begin{tikzpicture}
				\fill[blue!10](-3,0)--(3,0)--(3,2.78)--(-3,2.78);
				\fill[red!10](-3,0)--(3,0)--(3,-2.78)--(-3,-2.78);
				\draw[domain=-1.5:1.5,very thick,fill=red!10]plot({\x},{pow(3*\x*\x+1,0.5)});
				\draw[domain=-1.5:1.5,very thick,fill=blue!10]plot({\x},{-pow(3*\x*\x+1,0.5)});

				\draw[very thick,->](-3.15,0)--(3.15,0);
				
				\node[]at(0,2){$+$};
				\node[]at(0,0.5){$-$};
				\node[]at(0,-0.5){$+$};
				\node[]at(0,-2){$-$};
				\node[]at(2.5,0.3){$\operatorname{Re}k$};
				
				\filldraw [black](0,1)circle(1.8pt)node[above]{$\ii\kappa_0$};
				\filldraw [black](0,-1)circle(1.8pt)node[below]{$-\ii\kappa_0$};
			\end{tikzpicture}
			\caption{$k_0\in\ii \R$}
			\label{signReiTheta_a}
		\end{subfigure}
		\begin{subfigure}{0.45\textwidth}
			\begin{tikzpicture}
				\fill[red!10](-3,0)--(3,0)--(3,2.78)--(-3,2.78);
				\fill[blue!10](-3,0)--(3,0)--(3,-2.78)--(-3,-2.78);
				\fill[red!10](-1,0)--(-3,0)--(-3,-2.78)--(-1.9,-2.78);
				\draw[domain=-1.6:0,very thick,fill=red!10]plot({-pow(\x*\x+1,0.5)},{1.732*\x});
				\fill[blue!10](-1,0)--(-3,0)--(-3,2.78)--(-1.9,2.78);
				\draw[domain=0:1.6,very thick,fill=blue!10]plot({-pow(\x*\x+1,0.5)},{1.732*\x});
				\fill[blue!10](1,0)--(3,0)--(3,2.78)--(1.9,2.78);
				\draw[domain=-1.6:0,very thick,fill=red!10]plot({pow(\x*\x+1,0.5)},{1.732*\x});
				\fill[red!10](1,0)--(3,0)--(3,-2.78)--(1.9,-2.78);
				\draw[domain=0:1.6,very thick,fill=blue!10]plot({pow(\x*\x+1,0.5)},{1.732*\x});
				
				\draw[very thick,->](-3.15,0)--(3.15,0);
				
				\node[]at(0,1){$+$};
				\node[]at(0,-1){$-$};
				\node[]at(2,1){$-$};
				\node[]at(2,-1){$+$};
				\node[]at(-2,1){$-$};
				\node[]at(-2,-1){$+$};
				\node[]at(2.5,0.3){$\operatorname{Re}k$};
				
				\filldraw [black](1,0)circle(1.8pt)node[below]{$k_0$};
				\filldraw [black](-1,0)circle(1.8pt)node[below]{$-k_0$};
			\end{tikzpicture}
			\caption{$k_0\in\R$}
			\label{signReiTheta_b}
		\end{subfigure}
		\caption{The signatures of $\operatorname{Re}(\ii \theta)$ for different values of the stationary phase point $k_0\in \ii\R$ and $k_0\in\R$, respectively.}
	\end{figure}
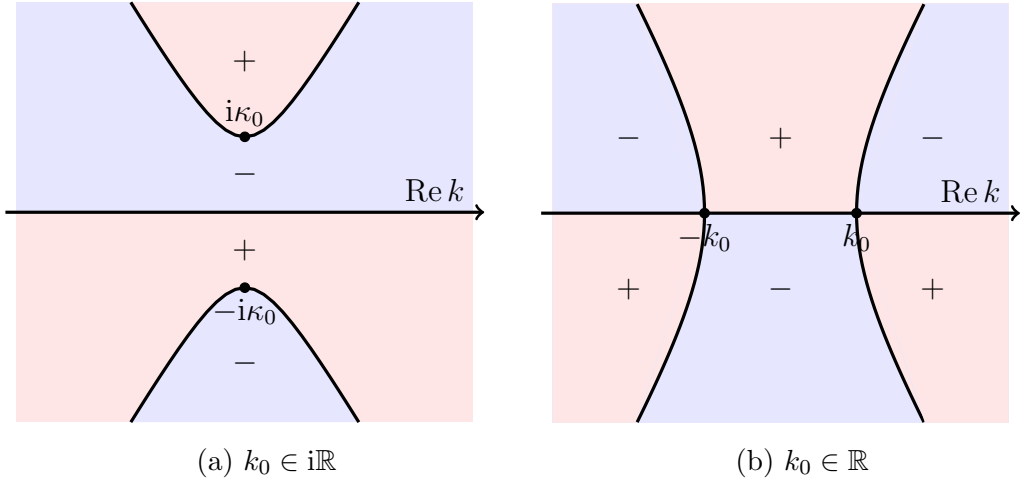

	If $\operatorname{Im}(k_p)>\kappa_0$, denote
	\[D_0(k)=
	\begin{pmatrix}
		T(k)^{-1}&0\\
		0&T(k)
	\end{pmatrix},\]
	where $T(k)=\dfrac{k+k_p}{k-k_p}$. If $0<\operatorname{Im}(k_p)<\kappa_0$, denote $D_0(k)=I$. So define $D$ by
	\begin{equation}
		D(k)=
		\begin{cases}
			\begin{pmatrix}
				1&-\dfrac{k-k_p}{\gamma_p\ee^{2\ii t\theta(k_p)}}\\
				\dfrac{\gamma_p\ee^{2\ii t\theta(k_p)}}{k-k_p}&0
			\end{pmatrix}D_0(k),&\quad k\in U_+,\\
			\begin{pmatrix}
				0&-\dfrac{\gamma_p\ee^{2\ii t\theta(k_p)}}{k+k_p}\\
				\dfrac{k+k_p}{\gamma_p\ee^{2\ii t\theta(k_p)}}&1
			\end{pmatrix}D_0(k),&\quad k\in U_-,\\
			D_0(k),&\quad\text{else}.
		\end{cases}
	\end{equation}
	for $\operatorname{Im}k_p>\kappa_0$, and define $D(k)=I$ for $0<\operatorname{Im}k_p<\kappa_0$. Obviously, $D(k)\to I$ as $k\to \infty$.
	\par
	Take the transformation
	$$
	M^{(1)}(x,t,k)=M^{(0)}(x,t,k)D(k),
	$$
	then the function $M^{(1)}(x,t,k)$ satisfies the following RH problem:
	\begin{RHP}
		Find a $2\times2$ matrix-valued function $M^{(1)}(x,t,k)$ with the following properties:\\
		(1) $M^{(1)}(x,t,\cdot)$ is sectionally analytic on $\C\backslash(\R\cup\Gamma_+\cup\Gamma_-)$ and it has countinus boundary value;\\
		(2) $M^{(1)}(x,t,k)=\sigma_1M^{(1)}(x,t,-k)\sigma_1$;\\
		(3) $M^{(1)}(x,t,k)\to I,\quad k\to\infty,k\notin\R$;\\
		(4) $M^{(1)}_{+}(x,t,k)=M^{(1)}_{-}(x,t,k)V^{(1)}(x,t,k)$ for $k\in \R\cup \Gamma_{+}\cup \Gamma_{-}$. When $\operatorname{Im}k_p>\kappa_0$, the jump matrix $V^{(1)}(x,t,k)$ is
		\begin{equation}
			V^{(1)}=
			\begin{cases}
				\begin{pmatrix}
					1&-\overline{r(k)}\ee^{-2\ii t\theta}T^2(k)\\
					0&1
				\end{pmatrix}
				\begin{pmatrix}
					1&0\\
					r(k)\ee^{2\ii t\theta}T^{-2}(k)&1
				\end{pmatrix}
				=:(b^{(1)}_{-})^{-1}b^{(1)}_{+},&\quad k\in\R,\\
				\begin{pmatrix}
					1&-\dfrac{k-k_p}{\gamma_p\ee^{2\ii t\theta}}T^{2}\\
					0&1
				\end{pmatrix},&\quad k\in\Gamma_{+},\\
				\begin{pmatrix}
					1&0\\
					-\dfrac{k+k_p}{\gamma_p\ee^{2\ii t\theta}}T^{-2}&1
				\end{pmatrix},&\quad k\in\Gamma_{-}.
			\end{cases}
		\end{equation}
		When $0<\operatorname{Im}k_p<\kappa_0$, the jump matrix $V^{(1)}(x,t,k)$ is
		\begin{equation}
			V^{(1)}(k)=V^{(0)}(k).
		\end{equation}
	\end{RHP}
	
	If $\operatorname{Im}k_p=\kappa_0$, since for all $\epsilon>0$, the jump matrixs on $\Gamma_{\pm}$ do not decay, retain the pole condition in this case:
	\begin{equation}
		\Res_{k=k_p}M^{(1)}(x,t,k)=\lim_{k\to k_p}M^{(1)}(x,t,k)
		\begin{pmatrix}
			0&0\\
			\gamma_{p}\ee^{2\ii t\theta(k_p)}&0
		\end{pmatrix},
		\label{res_M^{(1)}_at_kp}
	\end{equation}
	\begin{equation}
		\Res_{k=-k_p}M^{(1)}(x,t,k)=\lim_{k\to -k_p}M^{(1)}(x,t,k)
		\begin{pmatrix}
			0&-\gamma_{p}\ee^{2\ii t\theta(k_p)}\\
			0&0
		\end{pmatrix}.
		\label{res_M^{(1)}_at_-kp}
	\end{equation}
	\begin{Proposition}
		If the function $M^{(sol)}(x,t,k)$ solves the items (2), (3) and (5) in the RH problem $\ref{RHP:tildeM}$, it satisfies
		\begin{equation}
			M^{(sol)}(k)=I+\frac{1}{k-k_p}
			\begin{pmatrix}
				\alpha&0\\
				\beta&0
			\end{pmatrix}
			-\frac{1}{k_p+k}
			\begin{pmatrix}
				0&\beta\\
				0&\alpha
			\end{pmatrix},
			\label{Msol}
		\end{equation}
		where
		\[\alpha=\frac{\ii U_0\ee^{-2U_0x+2U_0^3t}}{\ee^{-2U_0x+2U_0^3t}-4},\quad \beta=\frac{2\ii U_0\ee^{-U_0x+U_0^3t}}{\ee^{-2U_0x+2U_0^3t}-4}.\]
		Then the simple soliton solution of the KdV equation is derived by the reconstruction formula in (\ref{u=F(tildeM)}) as
		\begin{equation}
			u_{sol}(x,t)=-\frac{U_0^2}{2} \operatorname{sech}^2\left(\frac{1}{2}(U_0x-U_0^3t+\log 2)\right).
		\end{equation}
	\end{Proposition}
	
	Then deform the RH problem for function $M^{(1)}(x,t,k)$ by the transformation
	\begin{equation}
		M^{(2)}(x,t,k)=
		\begin{cases}
			M^{(1)}(x,t,k)(b^{(1)}_{+}(k))^{-1},&0<\operatorname{Im}k<\epsilon,\\
			M^{(1)}(x,t,k)(b^{(1)}_{-}(k))^{-1},&-\epsilon<\operatorname{Im}k<0,\\
			M^{(1)}(x,t,k),&\text{others}.
		\end{cases}
	\end{equation}
	\par
	Then the function $M^{(2)}(x,t,k)$ satisfies the following RH problem:
	\begin{RHP}
		Find a $2\times2$ matrix-valued function $M^{(2)}(x,t,k)$ with the following properties:\\
		(1) $M^{(2)}(x,t,\cdot)$ is sectionally analytic on $\C\backslash(\Sigma_+\cup\Sigma_-\cup\Gamma_+\cup\Gamma_-)$ and it has countinus boundary value;\\
		(2) $M^{(2)}(x,t,k)=\sigma_1M^{(2)}(x,t,-k)\sigma_1$;\\
		(3) $M^{(2)}(x,t,k)\to I,\quad k\to\infty,\ k\notin\R$;\\
		(4) $M^{(2)}_{+}(x,t,k)=M^{(2)}_{-}(x,t,k)V^{(2)}(x,t,k)$ for $k\in \Sigma_+\cup\Sigma_-\cup\Gamma_+\cup\Gamma_-$ in Figure \ref{counter4.3}, where for $\operatorname{Im}k_p>\kappa_0$ the jump matrix is
		\begin{equation}
			V^{(2)}=
			\begin{cases}
				b^{(1)}_{+}(k),&k\in\Sigma_+,\\
				(b^{(1)}_{-}(k))^{-1}&k\in\Sigma_-,\\
				I,&k\in\R,\\
				V^{(1)},&k\in\Gamma_{\pm}.
			\end{cases}
		\end{equation}
	\end{RHP}
	\begin{figure}[H]
		\centering
		\begin{subfigure}{0.45\textwidth}
			\begin{tikzpicture}
				\draw[very thick,black!40!green,-latex](-3,0)--(0,0);
				\draw[very thick,black!40!green,-](-0.2,0)--(3,0);
				\draw[very thick,black!40!green,-latex](-3,0.5)--(0,0.5);
				\draw[very thick,black!40!green,-](-0.2,0.5)--(3,0.5);
				\draw[very thick,black!40!green,-latex](-3,-0.5)--(0,-0.5);
				\draw[very thick,black!40!green,-](-0.2,-0.5)--(3,-0.5);
				
				\draw[very thick,gray,dashed,domain=-1.5:1.5]plot({\x},{-pow(3*\x*\x+1,0.5)});
				\draw[very thick,gray,dashed,domain=-1.5:1.5]plot({\x},{pow(3*\x*\x+1,0.5)});
				
				\draw[black!20!blue, very thick] (0,1.5) circle (0.35);
				\draw[very thick,black!40!blue,-latex](0,1.15)--(-0.15,1.15);
				\filldraw [black] (0,1.5) circle (1pt);
				\draw[black!20!blue, very thick] (0,-1.5) circle (0.35);
				\draw[very thick,black!40!blue,-latex](0,-1.15)--(-0.15,-1.15);
				\filldraw [black] (0,-1.5) circle (1pt);
				
				\node[]at(3,0.7){$\Sigma_+$};
				\node[]at(3,-0.85){$\Sigma_-$};
				\node[]at(0,2){$\Gamma_+$};
				\node[]at(0,-2.2){$\Gamma_-$};
			\end{tikzpicture}
			\caption{$\operatorname{Im}k_p>\kappa_0$}
		\end{subfigure}
		\begin{subfigure}{0.45\textwidth}
			\begin{tikzpicture}
				\draw[very thick,black!40!green,-latex](-3,0)--(0,0);
				\draw[very thick,black!40!green,-](-0.2,0)--(3,0);
				\draw[very thick,black!40!green,-latex](-3,0.3)--(0,0.3);
				\draw[very thick,black!40!green,-](-0.2,0.3)--(3,0.3);
				\draw[very thick,black!40!green,-latex](-3,-0.3)--(0,-0.3);
				\draw[very thick,black!40!green,-](-0.2,-0.3)--(3,-0.3);
				
				\draw[very thick,gray,dashed,domain=-1.5:1.5]plot({\x},{-pow(3*\x*\x+1.5,0.5)});
				\draw[very thick,gray,dashed,domain=-1.5:1.5]plot({\x},{pow(3*\x*\x+1.5,0.5)});
				
				\draw[black!20!blue, very thick] (0,0.8) circle [radius=0.35];
				\draw[very thick,black!40!blue,-latex](0,0.45)--(-0.15,0.45);
				\filldraw [black] (0,0.8) circle (1pt);
				\draw[black!20!blue, very thick] (0,-0.8) circle [radius=0.35];
				\draw[very thick,black!40!blue,-latex](0,-0.45)--(-0.15,-0.45);
				\filldraw [black] (0,-0.8) circle (1pt);
				
				\node[]at(3,0.7){$\Sigma_+$};
				\node[]at(3,-0.85){$\Sigma_-$};
				\node[]at(0.8,0.8){$\Gamma_+$};
				\node[]at(0.8,-0.8){$\Gamma_-$};
			\end{tikzpicture}
			\caption{$0<\operatorname{Im}k_p<\kappa_0$}
		\end{subfigure}
		\caption{The contour for the RH problem with solution $M^{(2)}(x,t,k)$.}
		\label{counter4.3}
	\end{figure}
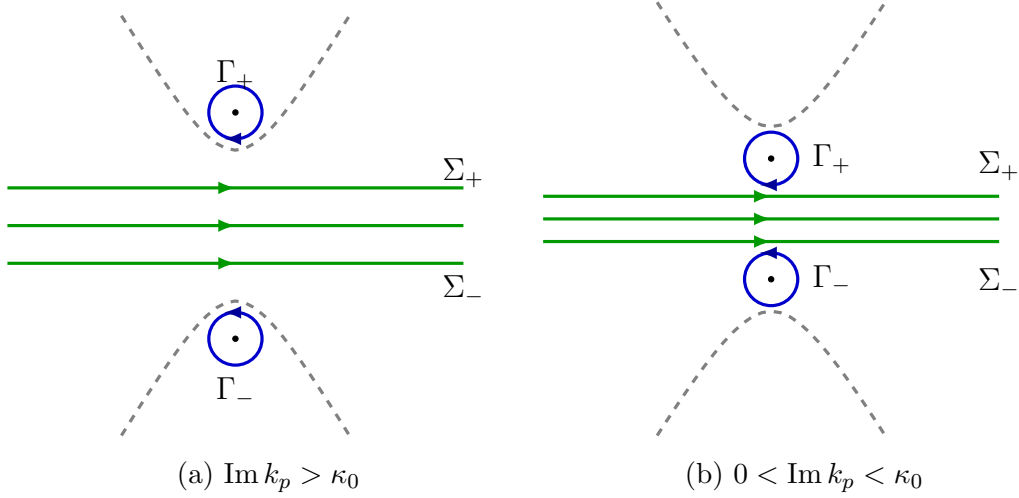
	
	\begin{theorem}
		Let $\epsilon>0$ be sufficiently small and let $l>0$ be an integer, then the long-time asymptotics of the solution to the initial problem (\ref{InitialProblem}) in the region of $x/t\geq C$ for some $C>0$ is given below:
		\par
		If $|\frac{x}{t}-U_0^2|<\epsilon$, one has
		\begin{equation}
			u(x,t)=-\frac{U_0^2}{2} \operatorname{sech}^2\left(\frac{1}{2}(U_0x-U_0^3t+\log 2)\right)+O(t^{-l}).
		\end{equation}
		\par
		If $|\frac{x}{t}-U_0^2|\geq\epsilon$, one has
		\begin{equation}
			u(x,t)=O(t^{-l}).
		\end{equation}
	\end{theorem}
	\begin{proof}
		The proof can be seen in the appendix of Ref. \cite{Grunert-Teschl-2009}.
	\end{proof}
	\subsubsection{Dispersive wave region}
	\label{U0>0:DispersiveReigon}
	When $x/t\geq -C$ for some $C>0$, the pole $k_p$ lies on the real axis. In this case, the signature of $\operatorname{Re}(\ii\theta)$ is shown in Figure $\ref{signReiTheta_b}$, thus introduce $T(k)$ by
	\begin{equation}
		T(k)=\frac{k+k_p}{k-k_p}\exp\left(\frac{1}{2\pi\ii}\int_{-k_0}^{k_0}\frac{\log(1-|r(s)|^2)}{s-k}\dd s \right)
		\label{T(k)}.
	\end{equation}
	Using the Plemelj's formulas, it can be checked that
	\[T_{+}(k)=T_{-}(k)(1-|r(k)|^2),\quad |k|<k_0.\]
	And denote
	\begin{equation}
		D(k)=
		\begin{cases}
			\begin{pmatrix}
				1&-\dfrac{k-k_p}{\gamma_p\ee^{2\ii t\theta(k_p)}}\\
				\dfrac{\gamma_p\ee^{2\ii t\theta(k_p)}}{k-k_p}&0
			\end{pmatrix}D_0(k),&\quad k\in U_{+},\operatorname{Im} k_p>\kappa_0,\\
			\begin{pmatrix}
				0&\dfrac{\gamma_p\ee^{2\ii t\theta(k_p)}}{k+k_p}\\
				-\dfrac{k+k_p}{\gamma_p\ee^{2\ii t\theta(k_p)}}&1
			\end{pmatrix}D_0(k),&\quad k\in U_{-},\operatorname{Im} k_p>\kappa_0,\\
			D_0(k),&\quad \quad\text{others},
		\end{cases}
	\end{equation}
	where $D_0(k)=(T(k))^{-\sigma_3}$. Take the transformation
	$$
	M^{(1)}(x,t,k)=M^{(0)}(x,t,k)D(k),
	$$
	then the function $M^{(1)}(x,t,k)$ satisfies the following RH Problem:
	
	\begin{RHP}
		Find a $2\times2$ matrix-valued function $M^{(1)}(x,t,k)$ with the following properties:\\
		(1) $M^{(1)}(x,t,\cdot)$ is sectionally analytic on $\C\backslash(\R\cup\Gamma_+\cup\Gamma_-)$ and it has countinus boundary value;\\
		(2) $M^{(1)}(x,t,k)=\sigma_1M^{(1)}(x,t,-k)\sigma_1$;\\
		(3) $M^{(1)}(x,t,k)\to I,\quad k\to\infty,k\notin\R$;\\
		(4) $M^{(1)}_{+}(x,t,k)=M^{(1)}_{-}(x,t,k)V^{(1)}(x,t,k)$ for $k\in \R\cup\Gamma_+\cup\Gamma_-$, where for $\operatorname{Im}k_p>\kappa_0$, the jump matrix is
		\begin{equation}
			V^{(1)}(k)=
			\begin{cases}
				\begin{pmatrix}
					1&\dfrac{k-k_p}{\gamma_p\ee^{2\ii t\theta(k_p)}}T^2\\
					0&1
				\end{pmatrix},&\quad k\in \Gamma_+,\\
				\begin{pmatrix}
					1&0\\
					-\dfrac{k+k_p}{\gamma_p\ee^{2\ii t\theta(k_p)}}T^{-2}&1
				\end{pmatrix},&\quad k\in \Gamma_-,\\
				(b^{(1)}_-(k))^{-1}b^{(1)}_{+}(k),&\quad |k|>k_0,\\
				(B^{(1)}_-(k))^{-1}B^{(1)}_{+}(k),&\quad |k|<k_0,
			\end{cases}
		\end{equation}
		where
		\[b^{(1)}_{-}(k)=\begin{pmatrix}
			1&\overline{r(k)}\ee^{-2\ii t\theta}T^2(k)\\
			0&1
		\end{pmatrix},\quad
		b^{(1)}_{+}(k)=\begin{pmatrix}
			1&0\\
			r(k)\ee^{2\ii t\theta}T^{-2}(k)&1
		\end{pmatrix},
		\]
		and
		\[B^{(1)}_{+}(k)=\begin{pmatrix}
			1&-\frac{\overline{r(k)}}{1-|r(k)|^2}\ee^{-2\ii t\theta(x,t)}T^2_{+}(k)\\
			0&1
		\end{pmatrix},\quad
		B^{(1)}_{-}(k)=\begin{pmatrix}
			1&0\\
			\frac{r(k)}{1-|r(k)|^2}\ee^{2\ii t\theta(x,t)}T^{-2}_{-}(k)&1
		\end{pmatrix},\]
		while for $\operatorname{Im}0<k_p<\kappa_0$, the jump matrix is
		\begin{equation}
			V^{(1)}(k)=
			\begin{cases}
				\begin{pmatrix}
					1&0\\
					-\dfrac{\gamma_p\ee^{2\ii t\theta(k_p)}}{k-k_p}T^{-2}&1
				\end{pmatrix},&\quad k\in \Gamma_+,\\\\
				\begin{pmatrix}
					1&-\dfrac{\gamma_p\ee^{2\ii t\theta(k_p)}}{k+k_p}T^{2}\\
					0&1
				\end{pmatrix},&\quad k\in \Gamma_-,\\
				(b^{(1)}_-(k))^{-1}b^{(1)}_{+}(k),&\quad |k|>k_0,\\
				(B^{(1)}_-(k))^{-1}B^{(1)}_{+}(k),&\quad |k|<k_0.
			\end{cases}
		\end{equation}
		\label{RH:usingTk}
	\end{RHP}
	
	Further, deform the RH problem with solution $M^{(1)}(x,t,k)$ by the following transformation
	\begin{equation}
		M^{(2)}(x,t,k)=
		\begin{cases}
			M^{(1)}(x,t,k)(b^{(1)}_{+}(k))^{-1},&\text{$k$ between $\R$ and $\Sigma^1_{+}$},\\
			M^{(1)}(x,t,k)(b^{(1)}_{-}(k))^{-1},&\text{$k$ between $\R$ and $\Sigma^1_{-}$},\\
			M^{(1)}(x,t,k)(B^{(1)}_{+}(k))^{-1},&\text{$k$ between $\R$ and $\Sigma^2_{+}$},\\
			M^{(1)}(x,t,k)(B^{(1)}_{-}(k))^{-1},&\text{$k$ between $\R$ and $\Sigma^2_{-}$},\\
			M^{(1)}(x,t,k),&	\text{others}.
		\end{cases}
	\end{equation}
	
	\begin{figure}[H]
		\centering
		\begin{tikzpicture}
			\draw[very thick,black!20!blue,->](-3.2,0)--(3.2,0);
			\draw[very thick,gray,dashed,domain=-1.633:1.633]plot({-pow(\x*\x+1,0.5)},{1.732*\x});
			\draw[very thick,gray,dashed,domain=-1.633:1.633]plot({pow(\x*\x+1,0.5)},{1.732*\x});
			
			\draw[very thick,black!40!green,-latex](-3,0.5)--(-2.25,0.5);
			\draw[very thick,black!40!green,-latex](-2.4,0.5)--(-1.5,0.5)--(-0.5,-0.5)--(0,-0.5);
			\draw[very thick,black!40!green,-latex](-0.2,-0.5)--(0.5,-0.5)--(1.5,0.5)--(2.25,0.5);
			\draw[very thick,black!40!green,-](2.05,0.5)--(3,0.5);
			\draw[very thick,black!40!green,-latex](-3,-0.5)--(-2.25,-0.5);
			\draw[very thick,black!40!green,-latex](-2.4,-0.5)--(-1.5,-0.5)--(-0.5,0.5)--(0,0.5);
			\draw[very thick,black!40!green,-latex](-0.5,0.5)--(0.5,0.5)--(1.5,-0.5)--(2.25,-0.5);
			\draw[very thick,black!40!green,-](2.05,-0.5)--(3,-0.5);
			
			\filldraw [black](1,0)circle(1.8pt)node[below]{$k_0$};
			\filldraw [black](-1,0)circle(1.8pt)node[below]{$-k_0$};
			
			\draw[black!20!blue, very thick] (0,1.8) circle [radius=0.35];
			\draw[very thick,black!40!blue,-latex](0,1.45)--(-0.15,1.45);
			\filldraw [black] (0,1.8) circle (1pt);
			\draw[black!20!blue, very thick] (0,-1.8) circle [radius=0.35];
			\draw[very thick,black!40!blue,-latex](0,-1.45)--(-0.15,-1.45);
			\filldraw [black] (0,-1.8) circle (1pt);
			
			\node[]at(2.25,0.8){$\Sigma^1_{+}$};
			\node[]at(2.25,-0.9){$\Sigma^1_{-}$};
			\node[]at(-2.25,0.8){$\Sigma^1_{+}$};
			\node[]at(-2.25,-0.9){$\Sigma^1_{-}$};
			\node[]at(0,0.8){$\Sigma^2_{+}$};
			\node[]at(0,-0.9){$\Sigma^2_{-}$};
			\node[]at(0,2.3){$\Gamma_+$};
			\node[]at(0,-2.5){$\Gamma_-$};
		\end{tikzpicture}
		\caption{The contour for RH problem with solution $M^{(2)}(x,t,k)$.}
		\label{counter-4.5}
	\end{figure}
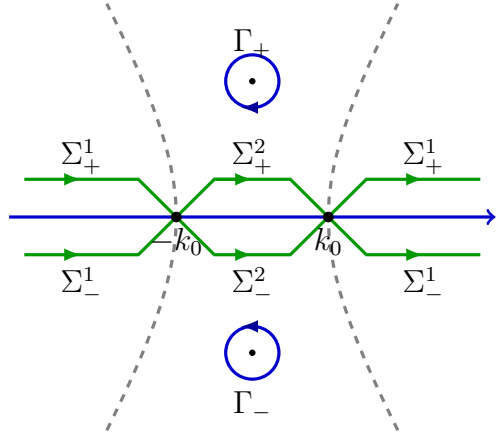
	\begin{RHP}
		Find a $2\times2$ matrix-valued function $M^{(2)}(x,t,k)$ with the following properties:\\
		(1) $M^{(2)}(x,t,\cdot)$ is sectionally analytic on $\C\backslash(\R\cup\Gamma_+\cup\Gamma_-)$ and it has countinus boundary value;\\
		(2) $M^{(2)}(x,t,k)=\sigma_1M^{(2)}(x,t,-k)\sigma_1$;\\
		(3) $M^{(2)}(x,t,k)\to I,\quad k\to\infty,k\notin\R$;\\
		(4) $M^{(2)}_{+}(x,t,k)=M^{(2)}_{-}(x,t,k)V^{(2)}(x,t,k)$ for $k\in \R\cup\Gamma_+\cup\Gamma_-$ in Figure \ref{counter-4.5},
		where the jump matrix is
		\begin{equation}
			V^{(2)}(k)=
			\begin{cases}
				b^{(1)}_{+}(k),&k\in\Sigma^1_{+},\\
				(b^{(1)}_{-}(k))^{-1},&k\in\Sigma^1_{-},\\
				B^{(1)}_{+}(k),&k\in\Sigma^2_{+},\\
				(B^{(1)}_{-}(k))^{-1},&k\in\Sigma^2_{-},\\
				V^{(1)}(k),&	\text{others}.
			\end{cases}
		\end{equation}
	\end{RHP}
	
	Denote the contours $\Sigma^c(\pm k_0)=(\Sigma^1_{+}\cup\Sigma^1_{-}\cup\Sigma^2_{+}\cup\Sigma^2_{-})\cap D(\pm k_0, \epsilon)$ for some small $\epsilon>0$. Let $M^{(\pm k_0)}(x,t,k)$ be two functions satisfying RH problem on contours $\Sigma^c(\pm k_0)$ with the same jump matrices as $M^{(2)}(x,t,k)$, then the functions $M^{(\pm k_0)}(x,t,k)$ solve the following RH problem:
	\begin{RHP}\label{RHP-4.6}
		Find a $2\times2$ matrix-valued function $M^{(\pm k_0)}(x,t,k)$ with the following properties:\\
		(1) $M^{(\pm k_0)}(x,t,\cdot)$ is sectionally analytic on $\C\backslash(\R\cup\Gamma_+\cup\Gamma_-)$ and it has continuous boundary value;\\
		(2) $M^{(\pm k_0)}(x,t,k)\to I,\quad k\to\infty,k\notin\R$;\\
		(3) $M^{(\pm k_0)}_{+}(x,t,k)=M^{(\pm k_0)}_{-}(x,t,k)V^{(\pm k_0)}(x,t,k)$ for $k\in \R\cup\Gamma_+\cup\Gamma_-$, where the jump matrix is
		\begin{equation}
			V^{(\pm k_0)}=
			\begin{cases}
				V^{(2)}(k),&k\in\Sigma^c(\pm k_0),\\
				I,&\text{others}.
			\end{cases}
		\end{equation}
	\end{RHP}
	
	In order to solve the RH problem \ref{RHP-4.6}, take the coordinate transformation
	\[\zeta=\sqrt{48k_0}(k-k_0),\quad k=k_0+\frac{\zeta}{\sqrt{48k_0}},\]
	then we have
	\[\ee^{2\ii t\theta(k)}=t\left( -16\ii k_0^3+\frac{\ii}{2}\zeta^2+\frac{8\ii\zeta^3}{48k_0\sqrt{48k_0}}\right), \]
	and
	\begin{equation}
		r(k)\ee^{2\ii t\theta}T(k)^{-2}=\ee^{2\ii t\theta(k(\zeta))}\zeta^{-2\ii\nu}r(k(\zeta))(\zeta+2k_0\sqrt{48k_0})^{2\ii\nu}\tilde{T}(k(\zeta)).
	\end{equation}
	where $\nu=-\frac{1}{2\pi}\log(1-|r(k_0)|^2)$.
	\par
	Rewrite $T(k)$ in Eq. $(\ref{T(k)})$ as
	\begin{equation}
		T(k)=\left(\frac{k-k_0}{k+k_0} \right)^{\ii\nu}\tilde{T}(k),
	\end{equation}
	where
	\begin{equation} \tilde{T}(k)=\frac{k+k_p}{k-k_p}\exp\left(\frac{1}{2\pi\ii}\int_{-k_0}^{k_0}\log\left(\frac{1-|r(s)|^2}{1-|r(k_0)|^2} \right)\frac{\dd s}{s-k}  \right).
	\end{equation}
	\par
	Define
	\begin{equation}
		r_0=r(k_0)\tilde{T}(k_0)^{-2}\ee^{2\ii\nu\log(2k_0\sqrt{48k_0})},
	\end{equation}
	then the following lemma holds:
	\begin{lemma}\label{Lemma-4.3}
		Assume $x/t\leq -C$ for constant $C>0$, then for all $0<\alpha<1$, there exist two constants $\delta_0>0$ and $L=L(C,\alpha,U_0)>0$ such that
		\begin{equation}\label{HolderCondition}
			\left|r(k(\zeta))(\zeta+2k_0\sqrt{48k_0})^{2\ii\nu}\tilde{T}(k(\zeta))-r_0 \right|\leq L|\zeta|^{\alpha} ,\quad |\zeta|<\delta_0.
		\end{equation}
	\end{lemma}
	\begin{proof}
		1. The triangle inequality shows that
		\begin{equation}
			\begin{aligned}
				&\left|r(k(\zeta))(\zeta+2k_0\sqrt{48k_0})^{2\ii\nu}\tilde{T}(k(\zeta))-r_0 \right|\\
				\leq&|\tilde{T}(k(\zeta))^{-2}-\tilde{T}(k_0)^{-2}||(\zeta+2k_0\sqrt{48k_0})^{2\ii\nu}r(k(\zeta))|\\
				&+|r(k(\zeta))\tilde{T}(k_0)^{-2}|\left|(\zeta+2k_0\sqrt{48k_0})^{2\ii\nu}-\ee^{2\ii\nu\log(2k_0\sqrt{48k_0})} \right| \\
				&+|\tilde{T}(k_0)^{-2}\ee^{2\ii\nu\log(2k_0\sqrt{48k_0})}||r(k(\zeta))-r(k_0)|.
			\end{aligned}
			\label{tri ineq}
		\end{equation}
		
		2. It is easily to know that $|\tilde{T}(k_0)|=1,\left| \ee^{2\ii\nu\log(2k_0\sqrt{48k_0})}\right|=1,|r(k(\zeta))|\leq 1$ and
		\begin{equation}
			|(\zeta+2k_0\sqrt{48k_0})^{2\ii\nu}|=|\ee^{-2\nu\arg(\zeta+2k_0\sqrt{48k_0})}|\leq \ee^{-2\nu\pi}\leq L',
			\label{r{k}holder}
		\end{equation}
		where $L'$ depends on constant $C$.
		
		3. Reminding $(\ref{r(k)})$, it is immediate that
		\begin{equation}
			\begin{aligned}
				|r(k(\zeta))-r(k_0)|&=\left|\frac{k_p}{k_p-k_0} \right|\frac{|\zeta|}{\sqrt{48k_0}} \frac{1}{|k_p-k_0-\zeta/\sqrt{48k_0}|}\\
				&\leq C'|\zeta|\frac{1}{||k_p-k_0|-|\zeta|/\sqrt{48k_0}|}.
			\end{aligned}
		\end{equation}
		\par
		Let $\delta_1=\sqrt{12k_0}|k_p-k_0|$, then for $|\zeta|<\delta_1$, one has
		\begin{equation}
			|r(k(\zeta))-r(k_0)|\leq C'|\zeta|\frac{2}{|k_p-k_0|}\leq C'|\zeta|,
			\label{r(k)-r(k0)}
		\end{equation}
		where constant $C'$ depends on $\alpha,\delta_1,U_0$ and $C$.
		\par
		The inequality $(\ref{r{k}holder})$ indicates that
		\begin{equation}
			\begin{aligned}
				&|\ee^{2\ii\nu\log(\zeta+2k_0\sqrt{48k_0})}-\ee^{2\ii\nu\log(2k_0\sqrt{48k_0})}|\\
				\leq&\left( \sup_{0\leq v\leq1}|\ee^{2\ii\nu\log(v\zeta+2k_0\sqrt{48k_0})}|\right) |2\ii\nu(\log(\zeta+2k_0\sqrt{48k_0})-\log(2k_0\sqrt{48k_0}))|\\
				\leq&L'|\log(\zeta+2k_0\sqrt{48k_0})-\log(2k_0\sqrt{48k_0})|\\
				\leq&C'\left(\sup_{0\leq v\leq1}|\log(v\zeta+2k_0\sqrt{48k_0})| \right)|\zeta|\\
				\leq& C'|\zeta| .
			\end{aligned}
		\end{equation}
		\par
		4. Finally, consider the term $|\tilde{T}(k(\zeta))^{-2}-\tilde{T}(k_0)^{-2}|$. Consider the term in $\tilde{T}(k)$ that contains the integral
		\begin{equation}
			\chi(k):=\frac{1}{2\pi\ii}\int_{-k_0}^{k_0}\log\left(\frac{1-|r(s)|^2}{1-|r(k_0)|^2} \right)\frac{\dd s}{s-k}.
		\end{equation}
		\par
		The expression of $r(k)$ in (\ref{r(k)}) shows that $|\chi(k_0)|<\infty$, we have
		\begin{equation}
			\begin{aligned}
				|\ee^{-2\chi(k(\zeta))}-\ee^{-2\chi(k_0)}|\leq&|\ee^{-2(\chi(k(\zeta))-\chi(k_0))}-1|\\
				\leq&\left( \sup_{0\leq v\leq1}\ee^{-2v(\chi(k(\zeta))-\chi(k_0))}\right) |\chi(k(\zeta))-\chi(k_0)|\\
				\leq& L'|\chi(k(\zeta))-\chi(k_0)|.
			\end{aligned}
			\label{e(chi)}
		\end{equation}
		Since
		\begin{equation}
			\log\left(\frac{1-|r(s)|^2}{1-|r(k_0)|^2} \right)\sim s-k_0,\quad s\to k_0,
		\end{equation}
		then integrating by parts yields that there is a $\delta_2>0$ such that
		\begin{equation}\label{inequality-4.37}
			|\chi(k(\zeta)-\chi(k_0))|\leq \frac{|\zeta|}{\sqrt{48k_0}}\left|\log\left| \frac{\zeta}{\sqrt{48k_0}}\right|  \right| ,\quad |\zeta|<\delta_2.
		\end{equation}
		Substituting the inequality (\ref{inequality-4.37}) into $(\ref{e(chi)})$, it follows that
		\begin{equation}
			|\ee^{-2\chi(k(\zeta))}-\ee^{-2\chi(k_0)}|\leq \frac{|\zeta|}{\sqrt{48k_0}}\left|\log\left| \frac{\zeta}{\sqrt{48k_0}}\right|  \right|\leq C'|\zeta|^\alpha,
			\label{e(chi)-e(chi0)}
		\end{equation}
		for all $0<\alpha<1$.
		\par
		Moreover, it is easy to see that $\frac{k(\zeta)-k_p}{k(\zeta)+k_p}$ is bounded in $D(0,\delta_3)$ and satisfies
		\begin{equation}
			\left| \frac{k(\zeta)-k_p}{k(\zeta)+k_p}-\frac{k_0-k_p}{k_0+k_p}\right| <L'|\zeta|,\quad |\zeta|<\delta_3,
			\label{ra(k)-ra(k0)}
		\end{equation}
		for certain small $\delta_3>0$.
		\par
		Then taking $\delta_0=\min\{\delta_1,\delta_2,\delta_3\}$ and substituting $(\ref{r(k)-r(k0)})$$(\ref{e(chi)-e(chi0)})$ and $(\ref{ra(k)-ra(k0)})$ into $(\ref{tri ineq})$, it can be gotten that
		\begin{equation}
			\left|r(k(\zeta))(\zeta+2k_0\sqrt{48k_0})^{2\ii\nu}\tilde{T}(k(\zeta))-r \right|\leq L|\zeta|^{\alpha} ,\quad |\zeta|<\delta_0,
		\end{equation}
		where $L$ depends on the constants $\alpha,C,U_0$ and $\delta_0$. This completes the proof of this lemma.
	\end{proof}
	\par
	By using Lemma \ref{Lemma-4.3} and Theorem \ref{PCmodle} in the appendix, it can be gotten that
	\begin{equation}
		M^{(k_0)}(x,t,k)=I+\frac{1}{\sqrt{48k_0}(k-k_0)}\frac{\ii}{t^{1/2}}
		\begin{pmatrix}
			0&-\beta\\
			\bar{\beta}&0
		\end{pmatrix}+O(t^{-\alpha}),
	\end{equation}
	where $\frac{1}{2}<\alpha<1$ and
	\begin{equation}
		\begin{aligned}
			\beta&=\sqrt{\nu}\ee^{\ii(\pi/4-\arg(r)+\arg(\Gamma(\ii\nu)))}\ee^{-2\ii t\theta(k_0)}t^{-\ii\nu}\\
			&=\sqrt{\nu}\ee^{\ii(\pi/4-\arg(r(k_0))+\arg(\Gamma(\ii\nu)))}(192k_0^3)^{\ii\nu}\tilde{T}(k_0)^2\ee^{-16\ii t k_0^3}t^{-\ii\nu}.
		\end{aligned}
	\end{equation}
	The symmetry $M^{(-k_0)}(x,t,k)=\sigma_1M^{(k_0)}(x,t,-k)\sigma_1$ shows that
	\begin{equation}
		M^{(-k_0)}(x,t,k)=I+\frac{1}{\sqrt{48k_0}(k+k_0)}\frac{\ii}{t^{1/2}}
		\begin{pmatrix}
			0&-\bar{\beta}\\
			\beta&0
		\end{pmatrix}+O(t^{-\alpha}).
	\end{equation}
	\par
	Thus the long-time asymptotics of the solution to the initial problem (\ref{InitialProblem}) can be given in the following theorem.
	
	\begin{theorem}
		In the dispersive wave region, i.e., $x/t\leq -C$ for some $C>0$, the asymptotic solution to the initial problem (\ref{InitialProblem}) is expressed by
		\begin{equation}
			u(x,t)=\sqrt{\frac{4\nu(k_0)k_0}{3t}}\sin(16tk_0^3-\nu(k_0)\log(192tk_0^3)+\phi(k_0))+O(t^{-\alpha}),
		\end{equation}
		where $k_0=\sqrt{-\dfrac{x}{12t}}$, $r(k)=\dfrac{\ii U_0}{\ii U_0-2k}$ and
		\begin{equation}
			\begin{aligned}
				\nu(k_0)&=\frac{1}{\pi}\log(1-|r(k_0)|^2),\\				\phi(k_0)&=\frac{\pi}{4}-\arg(r(k_0))+\arg(\Gamma(\ii\nu(k_0)))+4\arctan\left(\frac{U_0}{2k_0} \right)-2\ii\chi(k_0).
			\end{aligned}
		\end{equation}
	\end{theorem}
	\begin{proof}
		Introduce the function $M^{s}(x,t,k)$ by
		\begin{equation}
			M^s(x,t,k)=
			\begin{cases}
				M^{(2)}(x,t,k)(M^{(\pm k_0)}(x,t,k))^{-1},&|k\pm k_0|\leq \epsilon,\\
				M^{(2)}(x,t,k),&\text{others},
			\end{cases}
		\end{equation}
		then the jump matrix of the RH problem for $M^s(x,t,k)$ is given by
		\begin{equation}
			V^s(k)=
			\begin{cases}
				(M^{(\pm k_0)}(k))^{-1},&|k\pm k_0|=2\epsilon,\\
				M^{(\pm k_0)}(k)V^{(2)}(k)(M^{(\pm k_0)}(k))^{-1},&k\in \Sigma^{(2)},\epsilon<|k\pm k_0|<2\epsilon,\\
				I,&k\in \Sigma^{(2)},|k\pm k_0|<\epsilon,\\
				V^{(2)}(k),&\text{others}.
			\end{cases}
		\end{equation}
		\par
		It is easy to prove that the jump matrices on $\Sigma^s\backslash (D(k_0,2\epsilon)\cup D(-k_0,\epsilon))$ are exponentially close to $I$ as $t\to \infty$ in sense of $\L^2$ and $\L^\infty$ norms, and close to $I$ at the rate of $t^{-1/2}$ on $D(\pm k_0,2\epsilon)$. The Beals-Coifman theory \cite{Beals-Coifman-1984} shows that
		\begin{equation}
			\begin{aligned}
				M^s(k)&=I+\frac{1}{2\pi\ii}\int_{\Sigma^s}\frac{\mu^s(z) w^s(z)}{z-k}\dd z\\
				&=I+\sum_{j=\pm k_0}\frac{1}{2\pi\ii}\int_{D(j,2\epsilon)}\frac{\mu^s(z)(M^{(j)}(z)^{-1}-I)}{z-k}\dd z+O(\ee^{-ct})\\
				&=I+\frac{1}{t^{1/2}}\frac{1}{2\pi\ii}\sum_{j=\pm k_0}M^{(j)}_1\int_{D(j,2\epsilon)}\frac{1}{z-j}\frac{1}{z-k}+O(t^{-\alpha})\\
				&=I+\frac{1}{t^{1/2}}\left( \frac{M^{(k_0)}_1}{k-k_0}+\frac{M^{(-k_0)}_1}{k+k_0}\right)+ O(t^{-\alpha}).
			\end{aligned}
		\end{equation}
		\par
		Because the deformations of the RH problems above are local, taking the limit in the same direction as the imaginary axis in equation ($\ref{u=F(tildeM)}$) yields the asymptotic solution of the form
		\begin{equation}
			\begin{aligned}
				u(x,t)&=-2\ii\partial_x\lim_{k\to\infty}(kM^s(k))\\
				&=-2\ii\partial_x\left(\frac{1}{\sqrt{48k_0}}\frac{\ii}{t^{1/2}}(\bar{\beta}+\beta) \right)+O(\ee^{-ct})\\
				&=\partial_x\left(\frac{2}{\sqrt{3tk_0}}\operatorname{Re}(\beta) \right) +O(\ee^{-ct})\\				
				&=\sqrt{\frac{4\nu(k_0)k_0}{3t}}\sin(16tk_0^3-\nu(k_0)\log(192tk_0^3)+\phi(k_0))+O(t^{-\alpha}).
			\end{aligned}
		\end{equation}
	\end{proof}

	\subsubsection{Self-similar region}
	
	In the self-similar region, it is aimed to deform the RH problem \ref{RHP:tildeM} into Painlev\'e II model problem in the appendix. Denoting $\tau=tk_0^3$, first consider the case $k_0\in\R$ and $\tau\leq C$ for constant $C>0$, then it is seen that $k_0\leq C^{1/3}t^{-1/3}$.
	\par
	In this case, the presence of pole $k_p$ will make the discussion of the problem more complicated, so one can perform the following transformation to simplify the problem:
	\begin{equation}
		M^{(1)}_P(x,t,k)=\tilde{M}(x,t,k)(M^{(sol)}(x,t,k))^{-1},
	\end{equation}
	where the function $M^{(sol)}(x,t,k)$ is given by ($\ref{Msol}$) with the large $k$ asymptotics
	\[M^{(sol)}(x,t,k)=I+\frac{M^{(sol)}_1}{k}+O(k^{-2}).\]
	
	\begin{RHP}
		Find a $2\times2$ matrix-valued function $M^{(1)}_P(x,t,k)$ with the following properties:\\
		(1) $M^{(1)}_P(x,t,\cdot)$ is sectionally analytic in $\C\backslash\R$;\\
		(2) $M^{(1)}_P(x,t,k)=\sigma_1M^{(1)}_P(x,t,-k)\sigma_1$;\\
		(3) $M^{(1)}_P(x,t,k)\to I,\quad k\to\infty,k\notin\R$;\\
		(4) $M^{(1)}_{P+}(x,t,k)=M^{(1)}_{P-}(x,t,k)V^{(1)}_P(x,t,k)$ for $k\in\R$,
		where the jump matrix is
		\begin{equation}
			V^{(1)}_P(k)=
			\begin{pmatrix}
				1-|r(k)|^2&-\overline{r(k)}{\ee}^{-2\ii t\theta(x,t)}\\
				r(k){\ee}^{2\ii t\theta(x,t)}&1
			\end{pmatrix}.
		\end{equation}
		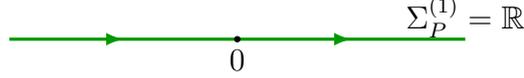
\begin{figure}[H]
			\centering
			\begin{tikzpicture}
				\draw[very thick,black!40!green,-latex](-3,0)--(-1.5,0);
				\draw[very thick,black!40!green,-latex](-1.7,0)--(1.5,0);
				\draw[very thick,black!40!green,-](1.3,0)--(3,0);
				\filldraw [black] (0,0) circle (1pt)node[below]{$0$};
				\node[]at(3,0.3){$\Sigma^{(1)}_P=\R$};
			\end{tikzpicture}
			\caption{The contour for the RH problem with solution $M^{(1)}_P(x,t,k)$.}
		\end{figure}
		\label{RHP:MP1}
	\end{RHP}
	Thus the solution of KdV equation can be expressed as
	\begin{equation}
		u(x,t)=u_{sol}(x,t)-2\ii\partial_x\left[\lim_{k\to\infty}k(M^{(1)}_{P,11}-1)+\lim_{k\to\infty}(kM^{(1)}_{P,21})\right].
	\end{equation}
	Specifically, since $\tau<C$, we further have
	\begin{equation}
		u(x,t)=-2\ii\partial_x\left[\lim_{k\to\infty}k(M^{(1)}_{P,11}-1)+\lim_{k\to\infty}(kM^{(1)}_{P,21})\right]+\ee^{-ct}.
	\end{equation}
	\par
	The next task is to deform the RH problem \ref{RHP:MP1}. To do so, denote
	\[b^{(1)}_{P+}(k)=\ee^{-\ii t\theta(x,t)\hat{\sigma}_3}
	\begin{pmatrix}
		1&0\\
		r(k)&1
	\end{pmatrix},\quad
	b^{(1)}_{P-}(k)=\ee^{-\ii t\theta(x,t)\hat{\sigma}_3}
	\begin{pmatrix}
		1&r^*(k)\\
		0&1
	\end{pmatrix},\]
	then $V^{(1)}_P(k)=(b^{(1)}_{P-}(k))^{-1}b^{(1)}_{P+}(k)$. Now scale $M^{(2)}_P(k)=M^{(1)}_P((3t)^{-1/3}k)$, then the jump matrix of $M^{(2)}_P(k)$ on $\R$ is $V^{(2)}_P(k)=V^{(1)}_P((3t)^{-1/3}k)$ and it follows that
	\[t\theta((3t)^{-1/3}k)=\frac{4}{3}k^3+sk,\quad s:=(3t)^{-1/3}x,\]
	\[b^{(2)}_{P\pm}(k)=b^{(1)}_{P\pm}((3t)^{-1/3}k).\]
	
	Then in order to open the jump contour outside the points $\pm C^{1/3}$ in Figure \ref{Jump-contour-Mp3}, take the transformation
	\begin{equation}
		M^{(3)}_P(x,t,k)=
		\begin{cases}
			M^{(2)}_P(x,t,k)(b^{(2)}_{P+}(k))^{-1},&k\in\Omega^{(3)}_{P1}\cup\Omega^{(3)}_{P3},\\
			M^{(2)}_P(x,t,k)(b^{(2)}_{P-}(k))^{-1},&k\in\Omega^{(3)}_{P4}\cup\Omega^{(3)}_{P6},\\
			M^{(2)}_P(x,t,k),               &k\in\Omega^{(3)}_{P2}\cup\Omega^{(3)}_{P2}.\\
		\end{cases}
	\end{equation}
	Thus the function $M^{(3)}_P(x,t,k)$ solves the RH problem below:
	\begin{RHP}
		Find a $2\times2$ matrix-valued function $M^{(3)}_P(x,t,k)$ with the following properties:\\
		(1) $M^{(3)}_P(x,t,\cdot)$ is sectionally analytic in $\C\backslash\Sigma^{(3)}_P$;\\
		(2) $M^{(3)}_P(x,t,k)=\sigma_1M^{(3)}_P(x,t,-k)\sigma_1$;\\
		(3) $M^{(3)}_P(x,t,k)\to I$ as $k\to\infty,\ k\notin\R$;\\
		(4) $M^{(3)}_{P+}(x,t,k)=M^{(3)}_{P-}(x,t,k)V^{(3)}_P(x,t,k)$ for $k\in \Sigma^{(3)}_P$, where the jump matrix is
		\begin{equation}
			V^{(3)}_P(k)=
			\begin{cases}
				b^{(2)}_{P+}(k),&k\in\Sigma^{(3)}_{P2}\cup\Sigma^{(3)}_{P3},\\
				(b^{(2)}_{P-}(k))^{-1},&k\in\Sigma^{(3)}_{P1}\cup\Sigma^{(3)}_{P4},\\
				V^{(2)}_P(k),&k\in\Sigma^{(3)}_{P5}.
			\end{cases}
		\end{equation}
	\end{RHP}
	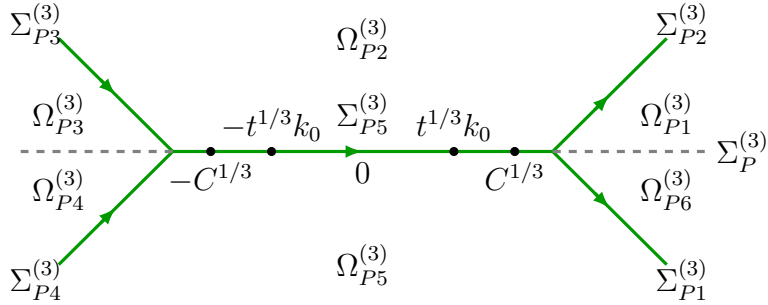
\begin{figure}[H]
		\centering
		\begin{tikzpicture}
			\draw[very thick,black!40!green,-latex](-4,1.5)--(-3.25,0.75);
			\draw[very thick,black!40!green,-](-3.45,0.95)--(-2.5,0);
			\draw[very thick,black!40!green,-latex](-2.5,0)--(0,0);
			\draw[very thick,black!40!green,-](-0.2,0)--(2.5,0);
			\draw[very thick,black!40!green,-latex](2.5,0)--(3.25,0.75);
			\draw[very thick,black!40!green,-](3.05,0.55)--(4,1.5);
			\draw[very thick,black!40!green,-latex](-4,-1.5)--(-3.25,-0.75);
			\draw[very thick,black!40!green,-](-3.45,-0.95)--(-2.5,0);
			\draw[very thick,black!40!green,-latex](2.5,0)--(3.25,-0.75);
			\draw[very thick,black!40!green,-](3.05,-0.55)--(4,-1.5);
			
			\draw[very thick,gray,dashed](2.5,0)--(4.5,0);
			\draw[very thick,gray,dashed](-2.5,0)--(-4.5,0);
			\node[]at(5,0){$\Sigma^{(3)}_P$};
			\node[]at(4.2,-1.7){$\Sigma^{(3)}_{P1}$};
			\node[]at(4.2,1.7){$\Sigma^{(3)}_{P2}$};
			\node[]at(-4.3,1.7){$\Sigma^{(3)}_{P3}$};
			\node[]at(-4.3,-1.7){$\Sigma^{(3)}_{P4}$};
			\node[]at(0,.5){$\Sigma^{(3)}_{P5}$};
			\filldraw [black] (2,0) circle (1.5pt)node[below]{$C^{1/3}$};
			\filldraw [black] (-2,0) circle (1.5pt)node[below]{$-C^{1/3}$};
			\filldraw [black] (1.2,0) circle (1.5pt)node[above]{$t^{1/3}k_0$};
			\filldraw [black] (-1.2,0) circle (1.5pt)node[above]{$-t^{1/3}k_0$};
			\filldraw [black] (0,0) circle (0pt)node[below]{$0$};
			\node[]at(4,0.5){$\Omega^{(3)}_{P1}$};
			\node[]at(0,1.5){$\Omega^{(3)}_{P2}$};
			\node[]at(-4,0.5){$\Omega^{(3)}_{P3}$};
			\node[]at(-4,-0.5){$\Omega^{(3)}_{P4}$};
			\node[]at(0,-1.5){$\Omega^{(3)}_{P5}$};
			\node[]at(4,-0.5){$\Omega^{(3)}_{P6}$};
		\end{tikzpicture}
		\caption{The contour for the RH problem with solution $M^{(3)}_P(x,t,k)$.}
		\label{Jump-contour-Mp3}
	\end{figure}

	Using the expansion $r(k)=r(0)+\int_{0}^{k}r'(\zeta)\dd\zeta$ for $k\in D(0,k_p/2)$ and $||r'||_{\L^{\infty}(D(0,k_p/2))}<\infty$, it can be obtained that
	\begin{equation}
		b^{(2)}_{P+}(k)=\ee^{-\ii (\frac{4}{3}k^3+sk)\hat{\sigma}_3}
		\begin{pmatrix}
			1&0\\
			r(0)&1
		\end{pmatrix}
		+\ee^{-\ii (\frac{4}{3}k^3+sk)\hat{\sigma}_3}
		\begin{pmatrix}
			1&0\\
			O(k/(3t)^{1/3})&1
		\end{pmatrix},
	\end{equation}
	then we have
	\begin{equation}
		\left\|\frac{k}{(3t)^{1/3}}\ee^{2\ii (\frac{4}{3}k^3+sk)} \right\|_{\L^1(\Sigma^{(3)}_{P2}\cup\Sigma^{(3)}_{P3})\cap\L^\infty(\Sigma^{(3)}_{P2}\cup\Sigma^{(3)}_{P3})}\leq ct^{-1/3},
	\end{equation}
	and the estimates for $V^{(3)}_P(k)$ on $\Sigma^{(3)}_{P1}\cup\Sigma^{(3)}_{P4}\cup\Sigma^{(3)}_{P5}$ can be given analogously.
	\begin{RHP}
		Find a $2\times2$ matrix-valued function $M^{(4)}_P(x,t,k)$ with the following properties:\\
		(1) $M^{(4)}_P(x,t,\cdot)$ is sectionally analytic in $\C\backslash\Sigma^{(4)}_P$, where $\Sigma^{(4)}_P=\Sigma^{(3)}_P$;\\
		(2) $M^{(4)}_P(x,t,k)=\sigma_1M^{(4)}_P(x,t,-k)\sigma_1$;\\
		(3) $M^{(4)}_P(x,t,k)\to I$ as $k\to\infty,\ k\notin\R$;\\
		(4) $M^{(4)}_{P+}(x,t,k)=M^{(4)}_{P-}(x,t,k)V^{(4)}_P(x,t,k)$ for $k\in \Sigma^{(4)}_P$, where the jump matrix is
		\begin{equation}
			V^{(4)}_P(k)=
			\begin{cases}
				b^{(4)}_{P+}(k)=\ee^{-\ii (\frac{4}{3}k^3+sk)\hat{\sigma}_3}
				\begin{pmatrix}
					1&0\\
					r(0)&1
				\end{pmatrix},&k\in\Sigma^{(4)}_{P2}\cup\Sigma^{(4)}_{P3},\\
				(b^{(4)}_{P-}(k))^{-1}=\ee^{-\ii (\frac{4}{3}k^3+sk)\hat{\sigma}_3}
				\begin{pmatrix}
					1&-r^*(0)\\
					0&1
				\end{pmatrix},&k\in\Sigma^{(4)}_{P1}\cup\Sigma^{(4)}_{P4},\\
				V^{(4)}_P(k)=\ee^{-\ii (\frac{4}{3}k^3+sk)\hat{\sigma}_3}
				\begin{pmatrix}
					1-|r(0)|^2&-r^*(0)\\
					r(0)&1
				\end{pmatrix},&k\in\Sigma^{(4)}_{P5}.
			\end{cases}
		\end{equation}
	\end{RHP}
	\par
	Take the transformation
	\begin{equation}
		M^{(5)}_P(x,t,k) =
		\begin{cases}
			M^{(4)}_P(x,t,k), & k \in \Omega_{P1}^{(5)} \cup \Omega_{P3}^{(5)} \cup \Omega_{P5}^{(5)} \cup \Omega_{7P}^{(5)}, \\
			M^{(4)}_P(x,t,k) (b_{P+}^{(4)})^{-1}, & k \in \Omega_{P2}^{(5)} \cup \Omega_{P4}^{(5)}, \\
			M^{(4)}_P(x,t,k) (b_{P-}^{(4)})^{-1}, & k \in \Omega_{P6}^{(5)} \cup \Omega_{P8}^{(5)},
		\end{cases}
	\end{equation}
	where the regions $\Omega^{(5)}_{Pj}~(j=1,2,\cdots,8)$ are described in Figure \ref{Contour-MP5}.
	\begin{figure}[H]
		\centering
		\begin{tikzpicture}
			\draw[dashed,very thick,gray,-latex](-4,1.5)--(-3.25,0.75);
			\draw[dashed,very thick,gray,-](-3.45,0.95)--(-2.5,0);
			\draw[dashed,very thick,gray,-latex](-2.5,0)--(0,0);
			\draw[dashed,very thick,gray,-](-0.2,0)--(2.5,0);
			\draw[dashed,very thick,gray,-latex](2.5,0)--(3.45,0.95);
			\draw[dashed,very thick,gray,-](3.25,0.75)--(4,1.5);
			\draw[dashed,very thick,gray,-latex](-4,-1.5)--(-3.25,-0.75);
			\draw[dashed,very thick,gray,-](-3.45,-0.95)--(-2.5,0);
			\draw[dashed,very thick,gray,-latex](2.5,0)--(3.45,-0.95);
			\draw[dashed,very thick,gray,-](3.25,-0.75)--(4,-1.5);
			
			\draw[very thick,black!40!green,-latex](-0.866,-1.5)--(-0.433,-0.75);
			\draw[very thick,black!40!green,-](-0.6,-1.05)--(0,0);
			\draw[very thick,black!40!green,-latex](0,0)--(0.6,1.05);
			\draw[very thick,black!40!green,-](0.433,0.75)--(0.866,1.5);
			\draw[very thick,black!40!green,-latex](-0.866,1.5)--(-0.433,0.75);
			\draw[very thick,black!40!green,-](-0.6,1.05)--(0,0);
			\draw[very thick,black!40!green,-latex](0,0)--(0.6,-1.05);
			\draw[very thick,black!40!green,-](0.433,-0.75)--(0.866,-1.5);
			
			\node[]at(5,0){$\Sigma^{(5)}_P$};
			\node[]at(4.2,1.7){$\Sigma^{(5)}_{P1}$};
			\node[]at(0.866,1.7){$\Sigma^{(5)}_{P2}$};
			\node[]at(-0.866,1.7){$\Sigma^{(5)}_{P3}$};
			\node[]at(-4.3,1.7){$\Sigma^{(5)}_{P4}$};
			\node[]at(-4.3,-1.7){$\Sigma^{(5)}_{P5}$};
			\node[]at(-0.866,-1.7){$\Sigma^{(5)}_{P6}$};
			\node[]at(0.866,-1.7){$\Sigma^{(5)}_{P7}$};
			\node[]at(4.2,-1.7){$\Sigma^{(5)}_{P8}$};
			
			\filldraw [black] (2,0) circle (1.5pt)node[below]{$C^{1/3}$};
			\filldraw [black] (-2,0) circle (1.5pt)node[below]{$-C^{1/3}$};
			\filldraw [black] (1.2,0) circle (1.5pt)node[above]{$t^{1/3}k_0$};
			\filldraw [black] (-1.2,0) circle (1.5pt)node[above]{$-t^{1/3}k_0$};
			\filldraw [black] (0,0) circle (0pt)node[below]{$0$};
			
			\node[]at(3.8,0){$\Omega^{(5)}_{P1}$};
			\node[]at(2,1.5){$\Omega^{(5)}_{P2}$};
			\node[]at(0,1.5){$\Omega^{(5)}_{P3}$};
			\node[]at(-2,1.5){$\Omega^{(5)}_{P4}$};
			\node[]at(-3.8,0){$\Omega^{(5)}_{P5}$};
			\node[]at(-2,-1.5){$\Omega^{(5)}_{P6}$};
			\node[]at(0,-1.5){$\Omega^{(5)}_{7}$};
			\node[]at(2,-1.5){$\Omega^{(5)}_{P8}$};
		\end{tikzpicture}
		\caption{The contour for RH problem of $M^{(5)}_P(x,t,k)$.}
		\label{Contour-MP5}
	\end{figure}
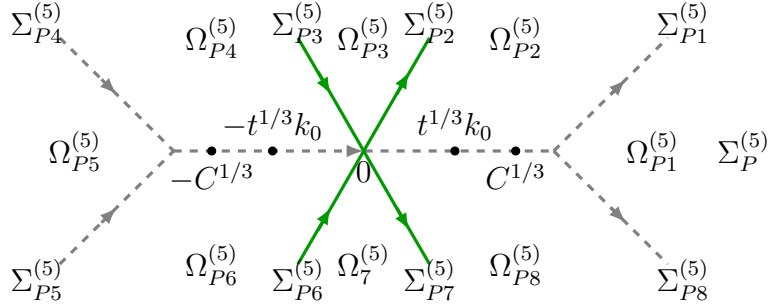
	After some calculations, it is found that the function $M^{(5)}_P(x,t,k)$ only has jumps across the contours $\Sigma^{(5)}_{Pj}~(j=2,3,6,7)$. So the RH problem for function $M^{(5)}_P(x,t,k)$ is obtained as:
	\begin{RHP}
		Find a $2\times2$ matrix-valued function $M^{(5)}_P(x,t,k)$ with the following properties:\\
		(1) $M^{(5)}_P(x,t,\cdot)$ is sectionally analytic in $\C\backslash(\Sigma^{(5)}_{P2}\cup\Sigma^{(5)}_{P3}\cup\Sigma^{(5)}_{P6}\cup\Sigma^{(5)}_{P7})$;\\
		(2) $M^{(5)}_P(x,t,k)=\sigma_1M^{(5)}_P(x,t,-k)\sigma_1$;\\
		(3) $M^{(5)}_P(x,t,k)\to I,\  k\to\infty$;\\
		(4) $M^{(5)}_{P+}(x,t,k)=M^{(5)}_{P-}(x,t,k)V^{(5)}_P(x,t,k)$ for $k\in \Sigma^{(5)}_{P2}\cup\Sigma^{(5)}_{P3}\cup\Sigma^{(5)}_{P6}\cup\Sigma^{(5)}_{P7}$ in Figure \ref{Contour-MP5}, where the jump matrix is
		\begin{equation}
			V^{(5)}_P(k)=
			\begin{cases}
				\ee^{-\ii (\frac{4}{3}k^3+sk)\hat{\sigma}_3}\begin{pmatrix}
					1&0\\
					r(0)&1
				\end{pmatrix},&k\in\Sigma^{(5)}_{P2}\cup\Sigma^{(5)}_{P3},\\\
				\ee^{-\ii (\frac{4}{3}k^3+sk)\hat{\sigma}_3}\begin{pmatrix}
					1&-r^*(0)\\
					0&1
				\end{pmatrix},&k\in\Sigma^{(5)}_{P6}\cup\Sigma^{(5)}_{P7}.\\
			\end{cases}
		\end{equation}
	\end{RHP}
	
	In order to match the jump matrix of the RH problem for $M^{(5)}_P(x,t,k)$ with the jump matrix of Theorem $\ref{Painleve-II model}$ in the appendix, perform the following transformation:
	\begin{equation}
		M^{(6)}_P(x,t,k)=
		\begin{cases}
			M^{(6)}_P(x,t,k),&k\in\Omega^{(6)}_{P1}\cup\Omega^{(6)}_{P2}\cup\Omega^{(6)}_{P4}\cup\Omega^{(6)}_{P5},\\
			M^{(6)}_P(x,t,k)V^{(5)}_P(k),&k\in\Omega^{(6)}_{P3},\\
			M^{(6)}_P(x,t,k)(V^{(5)}_P(k))^{-1},&k\in\Omega^{(6)}_{P6},
		\end{cases}
	\end{equation}
	then the $M^{(6)}(x,t,k)$ satisfies the conditions of Theorem \ref{Painleve-II model} with the parameters given by
	\[p=r(0),\ q=-r(0),\ r=0.\]
	
	Making use of the reconstruction formula in (\ref{u=F(tildeM)}), the long-time asymptotics of the solution to the initial problem (\ref{InitialProblem}) can be given in the following theorem.
	\begin{theorem} In the self-similar region, i.e., $\tau<C$ for some $C>0$, the asymptotic solution to the initial problem (\ref{InitialProblem}) can be described by the Painlev\'e II equation in the form
		\begin{equation}
			u(x,t)=\frac{1}{(3t)^{2/3}}\left( y^{2}\left( \frac{x}{(3t)^{1/3}}\right) +y'\left(\frac{x}{(3t)^{1/3}}\right)\right) +O(t^{-2/3}).
		\end{equation}
	\end{theorem}
	\begin{proof}
		The Theorem \ref{Painleve-II model} in the appendix indicates that
		\begin{equation}
			M^{(6)}(x,t,k)=I+\frac{\ii}{2k}
			\begin{pmatrix}
				\int_{s_0}^{s}y^2(\xi)\dd\xi&-y(s)\\
				y(s)&-\int_{s_0}^{s}y^2(\xi)\dd\xi
			\end{pmatrix}
			+O(k^{-2}),
		\end{equation}
		where $y(s)=y(s;p,q,r)$ is the solution of Painlev\'e II equation with Stokes constants $(p,q,r)$.   \par
		\par
		Then the reconstruction formula in (\ref{u=F(tildeM)}) shows that
		\begin{equation}
			\begin{aligned}
				u(x,t)&=-2\ii\frac{\partial}{\partial x}\left( \lim_{k\to\infty}(k(\tilde{M}_{11}(x,t,k)-1))+\lim_{k\to\infty}k\tilde{M}_{21}(x,t,k)\right) \\
				&=-2\ii\frac{\partial}{\partial x}\left( \lim_{k\to\infty}(k(M^{(1)}_{P,11}(x,t,k)-1))+\lim_{k\to\infty}kM^{(1)}_{P,21}(x,t,k)\right)+O(\ee^{-ct})\\
				&=-2\ii\frac{1}{(3t)^{1/3}}\frac{\partial}{\partial x}\left( \lim_{k\to\infty}(k(M^{(2)}_{P,11}(x,t,k)-1))+\lim_{k\to\infty}kM^{(2)}_{P,21}(x,t,k)\right)+O(\ee^{-ct})\\
				&=-2\ii\frac{1}{(3t)^{1/3}}\frac{\partial}{\partial x}\left( \lim_{k\to\infty}(k(M^{(4)}_{P,11}(x,t,k)-1))+\lim_{k\to\infty}kM^{(4)}_{P,21}(x,t,k)\right)+O(\ee^{-ct})+O(t^{-2/3})\\
				&=-2\ii\frac{1}{(3t)^{1/3}}\frac{\partial}{\partial x}\left( \lim_{k\to\infty}(k(M^{(6)}_{P,11}(x,t,k)-1))+\lim_{k\to\infty}kM^{(6)}_{P,21}(x,t,k)\right)+O(t^{-2/3})\\
				&=\frac{1}{(3t)^{1/3}}\frac{\partial}{\partial x}\left(\int_{s_0}^{s}y^2(\xi)\dd\xi +y(s)\right)+O(t^{-2/3})\\
				&=\frac{1}{(3t)^{2/3}}\left( y^{2}\left( \frac{x}{(3t)^{1/3}}\right) +y'(\frac{x}{(3t)^{1/3}})\right) +O(t^{-2/3}).
			\end{aligned}
		\end{equation}
	\end{proof}

	In what follow, consider the case of $k_0\in\ii\R$ and $|\tau|\leq C$ for constant $C>0$. Since $k_0\in\ii\R$, the deformation of $M^{(1)}_P$ will save many steps. Firstly, scale $M^{(7)}_P(x,t,k)=M^{(1)}_P(x,t,(3t)^{-1/3}k)$, where the RH problem for $M^{(7)}_P(x,t,k)$ has jump matrix of the form
	\begin{equation}
		V^{(7)}_P(k)=(b^{(7)}_{P-}(k))^{-1}b^{(7)}_{P+}(k), \quad k\in\Sigma^{(8)}_P:=\R.
	\end{equation}
	Similarly, we have
	\begin{equation}
		b^{(7)}_{P+}(k)=\ee^{-\ii (\frac{4}{3}k^3+sk)\hat{\sigma}_3}
		\begin{pmatrix}
			1&0\\
			r(0)&1
		\end{pmatrix}
		+\ee^{-\ii (\frac{4}{3}k^3+sk)\hat{\sigma}_3}
		\begin{pmatrix}
			1&0\\
			O(k/(3t)^{1/3})&1
		\end{pmatrix},
	\end{equation}
	and
	\begin{equation}
		\left\|\frac{k}{(3t)^{1/3}}\ee^{2\ii (\frac{4}{3}k^3+sk)} \right\|_{\L^1( \Sigma^{(2)}_{P}) \cap\L^\infty(\Sigma^{(2)}_{P})}\leq ct^{-1/3},
	\end{equation}
	
	Then by Beals-Coifman theory \cite{Beals-Coifman-1984}, we only need to consider the following RH problem:
	\begin{RHP}
		Find a $2\times2$ matrix-valued function $M^{(8)}_P(x,t,k)$ with the following properties:\\
		(1) $M^{(8)}_P(x,t,\cdot)$ is sectionally analytic in $\C\backslash\R$;\\
		(2) $M^{(8)}_P(x,t,k)=\sigma_1M^{(8)}_P(x,t,-k)\sigma_1$;\\
		(3) $M^{(8)}_P(x,t,k)\to I,\quad k\to\infty,k\notin\R$;\\
		(4) $M^{(8)}_{P+}(x,t,k)=M^{(8)}_{P-}(x,t,k)V^{(8)}_P(x,t,k)$ for $k\in \R$,
		where the jump matrix is
		\begin{equation}
			V^{(8)}_P(x,t,k)=(b^{(8)}_{P-}(k))^{-1}b^{(8)}_{P+}(k),
		\end{equation}
		with
		\begin{equation}
			b^{(8)}_{P+}(k):=\ee^{-\ii (\frac{4}{3}k^3+sk)\hat{\sigma}_3}
			\begin{pmatrix}
				1&0\\
				r(0)&1
			\end{pmatrix},\quad
			b^{(8)}_{P-}(k):=\ee^{-\ii (\frac{4}{3}k^3+sk)\hat{\sigma}_3}
			\begin{pmatrix}
				1&r^*(0)\\
				0&1
			\end{pmatrix}.
		\end{equation}
	\end{RHP}
	
	Following the similar procedures as the previous part, the long-time asymptotics of the solution to the initial problem (\ref{InitialProblem}) in the self-similar region $|\tau|<C$ for some $C>0$ can be given in the following theorem.
	
	\begin{theorem}
		In the region $k_0\in\ii\R$ and $|\tau|<C$ for some $C>0$, the asymptotic solution to the initial problem (\ref{InitialProblem}) can be described by the Painlev\'e II equation in the form
		\begin{equation}
			u(x,t)=\frac{1}{(3t)^{2/3}}\left( y^{2}\left( \frac{x}{(3t)^{1/3}}\right) +y'(\frac{x}{(3t)^{1/3}})\right) +O(t^{-2/3}).
		\end{equation}
	\end{theorem}

	\subsection{\bf Case II. The parameter $U_0<0$}
	It can be seen that when $U_0<0$, the solution $\tilde{M}(x,t,k)$ of the RH problem \ref{RHP:tildeM}  has no singularities on $\C\backslash\R$. That is to say, the soliton is absent in this case, so one can only discuss the dispersive wave region and self-similar region.
	
	\subsubsection{Dispersive wave region}
	In this region, one just needs to replace $T(k)$ in $(\ref{T(k)})$ with
	\begin{equation}
		T(k)=\exp\left(\frac{1}{2\pi\ii}\int_{-k_0}^{k_0}\frac{\log(1-|r(s)|^2)}{s-k}\dd s \right),
	\end{equation}
	and replace the jump matrix on $\Gamma_{\pm}$ with the identity matrix. Then repeating the steps in Section \ref{U0>0:DispersiveReigon}, the similar result can also be obtained:
	\begin{theorem}
		In the dispersive wave region, i.e., $x/t\leq -C$ for some $C>0$, the asymptotic solution to the initial problem (\ref{InitialProblem}) can be expressed by
		\begin{equation}\label{asymptotic-solution-DW-u0less0}
			u(x,t)=\sqrt{\frac{4\nu(k_0)k_0}{3t}}\sin(16tk_0^3-\nu(k_0)\log(192tk_0^3)+\phi(k_0))+O(t^{-\alpha}),
		\end{equation}
		where $k_0=\sqrt{-\dfrac{x}{12t}}$, $r(k_0)=\dfrac{\ii U_0}{\ii U_0-2k_0}$ and
		\begin{equation}
			\begin{aligned}
				\nu(k_0)&=\frac{1}{\pi}\log(1-|r(k_0)|^2),\\
				\phi(k_0)&=\frac{\pi}{4}-\arg(r(k_0))+\arg(\Gamma(\ii\nu(k_0)))-2\ii\chi(k_0). \\
			\end{aligned}
		\end{equation}
	\end{theorem}
	\par
	It is remarked from the asymptotic expression (\ref{asymptotic-solution-DW-u0less0}) that the soliton only affects the phase of the solution in the long-time sense.
	
	\subsubsection{Self-similar region}
	
	As can be seen from the situation of $U_0>0$, soliton has no influence on the self-similar region. So the same result for $U_0<0$ can also be derived:
	\begin{theorem}
		In the self-similar region, i.e., $|\tau|<C$ for some $C>0$, the asymptotic solution to the initial problem (\ref{InitialProblem}) can be expressed by
		\begin{equation}
			u(x,t)=\frac{1}{(3t)^{2/3}}\left( y^{2}\left( \frac{x}{(3t)^{1/3}}\right) +y'(\frac{x}{(3t)^{1/3}})\right) +O(t^{-2/3}).
		\end{equation}
	\end{theorem}

	\section{The general delta function initial profile with multi-spikes}\label{section-general}
	
	In this section, the initial data in (\ref{InitialProblem}) of KdV equation with one delta function is extended into the general delta function with multi-spikes \cite{Cohen-Kappeler-1987}, as given in (\ref{L-delta-intial}). Further assume that $x_1<x_2<\cdots<x_L$, which denotes that the spikes of the general delta function locate at different points.
	\par
	Substituting (\ref{L-delta-intial}) into the first equation of Volterra integral equations (\ref{volterra-integral-equations}), it is obtained that
	\begin{equation*}
		N_l(x,k)=I-\sum_{n=1}^L\frac{\ii U_n}{2k}\int_{-\infty}^x\delta_n(s)\ee^{-\ii k(x-s)\hat{\sigma}_3}(QN_l(s,k))\dd s.
	\end{equation*}
	\par
	When $x< x_1$, it follows that
	\begin{equation}
		\label{N_l-L-delta-1}
		N_l(x,k)=I.
	\end{equation}
	\par
	When $x_m\leq x\leq x_{m+1},\ m=1,2,\cdots,L$, it follows that
	\begin{equation}
		\label{N_l-L-delta-2-N}
		\begin{aligned}
			N_l(x,k)&=I-\left( \sum_{n=1}^m+\sum_{n=m+1}^L\right) \frac{\ii U_n}{2k}\int_{-\infty}^x\delta_n(s)\ee^{-\ii k(x-s)\hat{\sigma}_3}(QN_l(s,k))\dd s\\
			&=I-\sum_{n=1}^m\frac{\ii U_n}{2k}\int_{-\infty}^x\delta_n(s)\ee^{-\ii k(x-s)\hat{\sigma}_3}(QN_l(s,k))\dd s\\
			&=I-\sum_{n=1}^m\frac{\ii U_n}{2k}\ee^{-\ii k(x-x_n)\hat{\sigma}_3}(QN_l(x_n,k)).
		\end{aligned}
	\end{equation}
	\par
	As long as one derives $N_l(x_n,k)$, then $N_l(x,k)$ can be obtained. Thus substituting $x=x_{m+1}$ for $m=1,2,\cdots,L-1$ into equation (\ref{N_l-L-delta-2-N}) yields
	\begin{equation*}
		N_l(x_{m+1},k)=I-\sum_{n=1}^m\frac{\ii U_n}{2k}\ee^{-\ii k(x_{m+1}-x_n)\hat{\sigma}_3}(QN_l(x_n,k)),
	\end{equation*}
	which can be rewritten as
	\begin{equation*}
		\ee^{\ii kx_{m}\hat{\sigma}_3}N_l(x_{m},k)=I-\sum_{n=1}^{m-1}\frac{\ii U_n}{2k}(\ee^{\ii kx_n\hat{\sigma}_3}Q)(\ee^{\ii kx_n\hat{\sigma}_3}N_l(x_n,k)),\ m=2,\cdots,L.
	\end{equation*}
	From (\ref{N_l-L-delta-1}), it follows that $N_l(x_1,k)=I$, which implies that $N_l(x_m,k)$ can be obtained by $N_l(x_1,k),\cdots,N_l(x_{m-1},k)$.
	\par	
	Similarly, it can be derived that $N_r(x,k)=I$ for $x>x_L$. Denote $S_L(k)$ as the scattering matrix for the initial condition $u(x,0)=-\sum_{n=1}^LU_n\delta_n(x)$, then the equation (\ref{scattering-matrax}) implies that
	\begin{equation}
		\label{S_L-1}
		\begin{aligned}
			S_L(k)&=\ee^{\ii kx\hat{\sigma}_3}(N_r^{-1}(x,k)N_l(x,k))\\
			&=\ee^{\ii kx\hat{\sigma}_3}N_l(x,k)\\
			&=I-\sum_{n=1}^L\frac{\ii U_n}{2k}(\ee^{\ii kx_n\hat{\sigma}_3}Q)(\ee^{\ii kx_n\hat{\sigma}_3}N_l(x_n,k))
		\end{aligned}
	\end{equation}
	for $x>x_L$. It can be proved that (\ref{S_L-1}) still holds for $x\leq x_L$.
	\par	
	In fact, be means of mathematical induction one can obtain that
	\begin{equation*}
		\ee^{\ii kx_{m+1}\hat{\sigma}_3}N_l(x_{m+1},k)=\left(I-\frac{\ii U_m}{2k}\ee^{\ii kx_m\hat{\sigma}_3}Q \right) \left(I-\frac{\ii U_{m-1}}{2k}\ee^{\ii kx_{m-1}\hat{\sigma}_3}Q \right) \cdots\left(I-\frac{\ii U_1}{2k}\ee^{\ii kx_1\hat{\sigma}_3}Q \right) ,
	\end{equation*}
	and
	\begin{equation}
		\label{S_L}
		S_L(k)=\left(I-\frac{\ii U_L}{2k}\ee^{\ii kx_L\hat{\sigma}_3}Q \right) \left(I-\frac{\ii U_{L-1}}{2k}\ee^{\ii kx_{L-1}\hat{\sigma}_3}Q \right)\cdots \left(I-\frac{\ii U_1}{2k}\ee^{\ii kx_1\hat{\sigma}_3}Q \right) .
	\end{equation}
	The recursive formula for reflection coefficient $r_L(k)$ can also be obtained
	\begin{equation}\label{R_L-L}
		r_L(k)=\frac{(2k+\ii U_L)r_{L-1}(k)-\ii U_L\ee^{-2\ii k x_L}}{\ii U_L\ee^{2\ii k x_L}r_{L-1}(k)+2k-\ii U_L},\quad L=2,3,\cdots,
	\end{equation}
	in which the first term $r_1(k)$ will be given below.

	\subsection{Potential of one delta function }
	
	For $L=1$ in initial data (\ref{L-delta-intial}), the associated scattering matrix and reflection coefficient are
	\begin{equation*}
		S_1(k)=I-\frac{\ii U_1}{2k}\ee^{\ii kx_1\hat{\sigma}_3}Q,
	\end{equation*}
	\begin{equation*}
		r_1(k)=\frac{\ii U_1}{\ii U_1-2k}\ee^{-2\ii k x_1},
	\end{equation*}
	which can recover the reflection coefficient in (\ref{r-k-one-delta}) for $U_1=U_0$ and $x_1=0$.
	In order to use the RH problem to analyze the long-time asymptotic behavior of the solution to the KdV equation in the initial data (\ref{L-delta-intial}) for $L=1$, it can be found through simple calculations that only the phase function $\theta(x,t,k)$ in (\ref{theta}) needs to be replaced with
	\begin{equation*}
		\theta(x,t,k)=k\frac{x-x_1}{t}+4k^3
	\end{equation*}
	in the RH problem \ref{RHP-3.1}. In this case, the analysis in Section \ref{section-proofs} remains unchanged, and it is only necessary to shift the corresponding long-time asymptotic results to the right by constant $x_1$.
	
	\subsection{Potential of double delta functions}
	
	For $L=2$ in initial data (\ref{L-delta-intial}), the associated scattering data are
	\begin{equation*}
		S_2(k)=I-\frac{\ii U_1}{2k}\ee^{\ii kx_1\hat{\sigma}_3}Q-\frac{\ii U_2}{2k}\ee^{\ii kx_2\hat{\sigma}_3}Q+\frac{(\ii U_1)(\ii U_2)}{(2k)^2}(\ee^{\ii kx_2\hat{\sigma}_3}Q)(\ee^{\ii kx_1\hat{\sigma}_3}Q),
	\end{equation*}
	\begin{equation*}
		s_{11}(k)=1-\frac{\ii U_1}{2k}-\frac{\ii U_2}{2k}+\frac{(\ii U_1)(\ii U_2)}{(2k)^2}(1-\ee^{2\ii k(x_2-x_1)}),
	\end{equation*}
	\begin{equation*}
		s_{21}(k)=-\frac{\ii U_1}{2k}-\frac{\ii U_2}{2k}+\frac{(\ii U_1)(\ii U_2)}{(2k)^2}(\ee^{-2\ii kx_2}-\ee^{-2\ii kx_1}),
	\end{equation*}
	and
	\begin{equation}\label{R_L-L=2}
		r_2(k)=\frac{-2\ii(U_1+U_2)k-U_1U_2(\ee^{-2\ii kx_2}-\ee^{-2\ii kx_1})}{4k^2-2\ii(U_1+U_2)k-U_1U_2(1-\ee^{2\ii k(x_2-x_1)})}.
	\end{equation}
	\par	
	The zeros of spectral function $s_{11}(k)$ provide discrete eigenvalues and lead to the emergence of solitons in KdV equation. Consider the equation $s_{11}(k)=0$, which is equivalent to $s_{11}(\ii z)=0$ for $k=\ii z$. Thus
	it follows that
	\begin{equation}
		\label{L=2-zero-equation}
		(2z-U_1)(2z-U_2)=U_1U_2\ee^{2(x_1-x_2)z}.
	\end{equation}
	\par
	Define $\sharp_L$ to be the number of discrete eigenvalues for initial data (\ref{L-delta-intial}). In the case of $L=2$, analyzing the positive roots of equation (\ref{L=2-zero-equation}) yields four cases:
	\begin{itemize}
		\item [(a)] If $U_1>0,\ U_2>0$ and $x_2-x_1>1/U_1+1/U_2$, then $\sharp_2=2$;
		\item [(b)] If $U_1>0,\ U_2>0$ and $x_2-x_1<1/U_1+1/U_2$, then $\sharp_2=1$;
		\item [(c)] If $U_1U_2<0$ and $x_2-x_1>1/U_1+1/U_2$, then $\sharp_2=1$;
		\item [(d)] If $U_1, U_2, x_1$ and $x_2$ don't satisfy the conditions in (a)-(c), then $\sharp_2=0$.
	\end{itemize}
	\par	
	Assume that $s_{11}(k)$ vanishes at two different points $k_j=\ii z_j\in\R_{+}$ and $z_1<z_2$. The long-time asymptotic of the KdV equation for the initial value $u(x,0)=-U_1\delta_1(x)-U_2\delta_2(x)$ in soliton region can be given as follows
	\begin{equation*}
		u(x,t)\sim-2\sum_{j=1}^2z_j^2\operatorname{sech}^2\left(\frac{1}{2}(z_jx-z_j^3t+l_j) \right) ,
	\end{equation*}
	where
	\begin{equation*}
		l_j=\frac{1}{2}\log\left(\frac{\gamma_j}{2z_j}\prod_{m=j+1}^2\left(\frac{z_m-z_j}{z_m+z_j} \right)  \right) ,\quad \gamma_j=\frac{s_{21}(k_j)}{s_{11}'(k_j)}.
	\end{equation*}
	\par
	\begin{figure}
		\centering
		\includegraphics[scale=0.4]{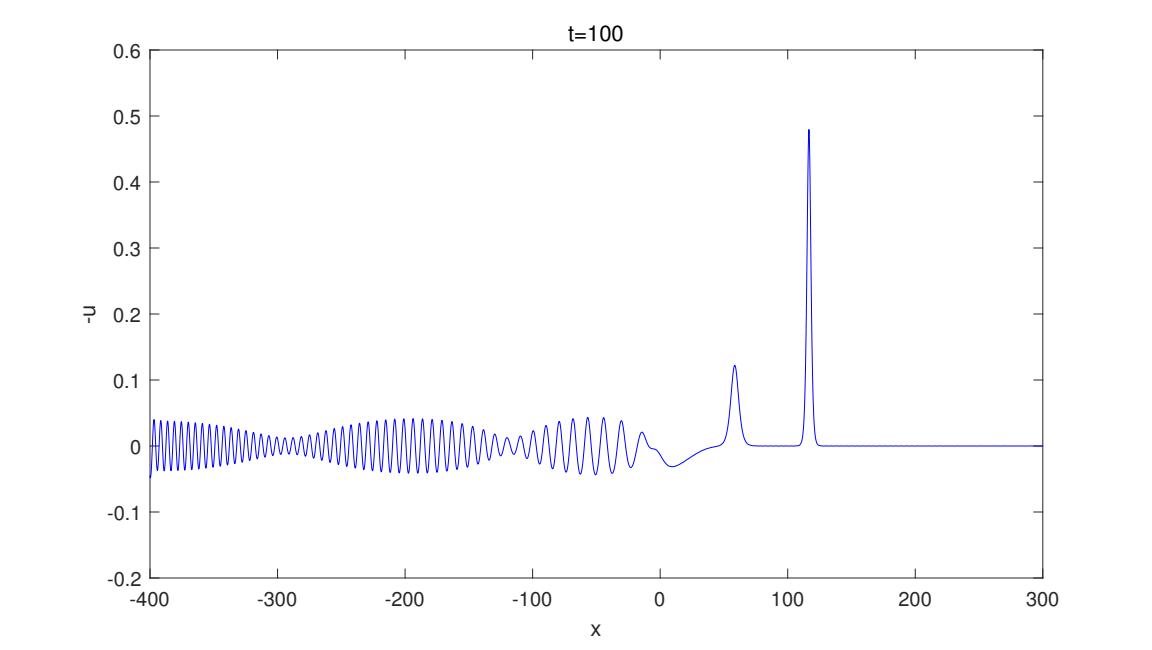}
		\caption{Numerical simulation of $-u(x,t)$ at $t=100$ for $L=2$, $U_1=U_2=0.5$, $x_1=20$ and $x_2=40$ in (\ref{L-delta-intial}).}
		\label{L2,sig=20,h=0.5,T=100}
	\end{figure}
	\par
	In particular, for $L=2$, take $U_1=U_2=0.5$, $x_1=20$ and $x_2=40$, which satisfies case (a) above, so two solitons emerge. Figure \ref{L2,sig=20,h=0.5,T=100} shows the numerical result of $-u(x,t)$ at $t=100$, where $u(x,t)$ solves the KdV equation. It is observed that for time large enough, the asymptotic solution of the KdV equation with delta function initial profile of two spikes can split into dispersive waves and two solitons.

	\subsection{Potential of $L$ delta functions}

	\begin{figure}
		\centering
		\begin{tikzpicture}
			\draw[black,very thick,-latex](-0.5,0)--(8,0);
			\draw[black,very thick,-latex](0,-1)--(0,4);
			
			\draw[domain=0:8]plot({\x},{(2*\x+3)/(\x+1)});
			\draw[domain=1.5:8]plot({\x},{(2*(\x-1.5))/\x});
			\draw[domain=3:8]plot({\x},{(2*(\x-3))/\x});
			\draw[domain=5:8]plot({\x},{(2*(\x-5))/\x});
			
			\draw[dashed](0,2)--(8,2);
			
			\draw[dashed](1,0)--(1,4);
			\node[]at(1,4){$\sharp_L=1$};
			\draw[dashed](2.25,0)--(2.25,4);
			\node[]at(2.25,4){$\sharp_L=2$};
			\draw[dashed](4,0)--(4,4);
			\node[]at(4,4){$\sharp_L=3$};
			\draw[dashed](6.5,0)--(6.5,4);
			\node[]at(6.5,4){$\sharp_L=L$};
			
			\node[]at(1.5,-0.6){$\frac{(\sigma h)_{L-1}^L}{h}$};
			\filldraw [black] (1.5,0) circle (1.5pt);
			\node[]at(4,-0.6){$\cdots$};
			\node[]at(3,-0.6){$\frac{(\sigma h)_{L-2}^L}{h}$};
			\filldraw [black] (3,0) circle (1.5pt);
			\node[]at(5,-0.6){$\frac{(\sigma h)_{1}^L}{h}$};
			\filldraw [black] (5,0) circle (1.5pt);
			
			\node[]at(-0.3,2){$\frac{\ii  h}{2}$};
			\filldraw [black] (0,2) circle (1.5pt);
			\node[]at(-0.3,3){$\frac{\ii L h}{2}$};
			\filldraw [black] (0,3) circle (1.5pt);
			
			\node[]at(8,0.2){$\sigma$};
			\node[]at(0,4.2){$k/\ii$};
			\node[]at(-0.2,-0.2){$0$};
			\filldraw [black] (0,0) circle (1.5pt);
		\end{tikzpicture}
		\caption{Schematic diagram of the number of discrete eigenvalues.}
		\label{sigma-k-plane}
	\end{figure}
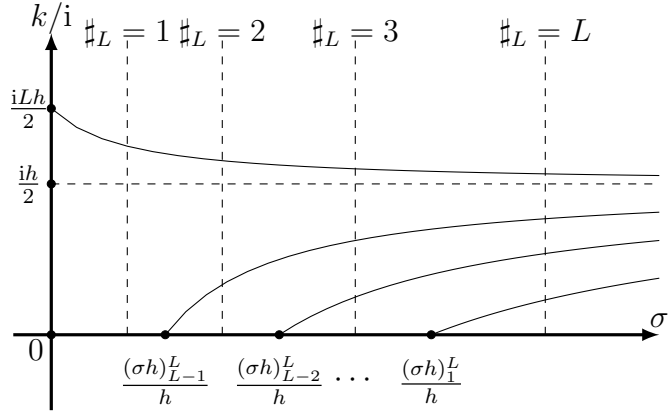

	In this part, it is proved that for the potential of $L\ (L\geq2)$ delta functions as given in (\ref{L-delta-intial}), the initial-value problem of KdV equation will produce solitons between one and $L$. For special parameters in (\ref{L-delta-intial}), the number of the solitons is classified by the following theorem.
	
	\begin{theorem}
		Let $h>0$ and $\sigma>0$. Taking $U_n=h$ and $x_n=n\sigma\ (n=1,2,\cdots,L)$ in initial data (\ref{L-delta-intial}), then the number of discrete eigenvalues, i.e., the number of solitons emerging from initial data (\ref{L-delta-intial}) is
		\begin{equation*}
			\sharp_L=L-j \quad \text{for}\quad  (\sigma h)_{j+1}^L<\sigma h\leq (\sigma h)_{j}^L,\ j=0,1,\cdots,L-1,
		\end{equation*}
		where $(\sigma h)_l^L=2+2\cos\left(\frac{l\pi}{L} \right),\ l=1,2,\cdots,L$ and $(\sigma h)^L_0=+\infty$. For $j=0$, ‘$\leq$' should be replaced by ‘$<$' (see Figure \ref{sigma-k-plane}).
	\end{theorem}
	\begin{proof}
		In order to discuss the number of discrete eigenvalues under initial condition (\ref{L-delta-intial}), fix parameter $\sigma$ in the following proof. Denote $a_L(k)=(S_L(k))_{11}$ and $b_L(k)=(S_L(k))_{21}$. From the expression of $S_L(k)$ in (\ref{S_L}), one can get the recursive formulas for $a_L(k)$ and $b_L(k)$ of the forms
		\begin{equation}
			\label{a_L+1}
			a_{L+1}(k)=\left(1-\frac{\ii h}{2k} \right)a_{L}(k)+\frac{\ii h}{2k}\ee^{2\ii k(L+1)\sigma}b_{L}(k) ,
		\end{equation}
		\begin{equation}
			\label{b_L+1}
			b_{L+1}(k)=-\frac{\ii h}{2k}\ee^{-2\ii k(L+1)\sigma}a_{L}(k)+\left(1+\frac{\ii h}{2k} \right)b_{L}(k) .
		\end{equation}
		
		It can be obtained that
		\begin{equation*}
			\lim_{\sigma\to0}S_L(k)=\left(I-\frac{\ii h}{2k}Q \right)^L=I-\frac{\ii Lh}{2k}Q
		\end{equation*}
		by using (\ref{S_L}) and that fact $Q^2=O$, which indicates that
		\begin{equation*}
			\lim_{\sigma\to0}a_L(k)=1-\frac{\ii Lh}{2k}.
		\end{equation*}
		Letting $\lim_{\sigma\to0}a_L(k)=0$ yields $k=\frac{\ii L h}{2}$.
		
		Next, consider the following limit
		\begin{equation*}
			A_L=\lim_{k\to0}\frac{2k}{\ii h}a_L(k).
		\end{equation*}
		By using (\ref{a_L+1}) and (\ref{b_L+1}) one can obtain that
		\begin{equation}
			\label{dituishi-A_L}
			A_{L+1}=A_L+B_L,
		\end{equation}
		where
		\begin{equation*}
			B_L=\lim_{k\to0}\left(\ee^{2\ii k(L+1)\sigma}b_L(k)-a_L(k) \right).
		\end{equation*}
		Also the recurrence formula for $B_L$ can be gotten as
		\begin{equation}
			\label{dituishi-B_L}
			B_{L+1}=(1-\sigma h)B_L-\sigma h A_L.
		\end{equation}
		
		By using equations (\ref{dituishi-A_L}) and (\ref{dituishi-B_L}) and performing simple calculations, the recurrence formula for $A_L$ is derived below
		\begin{equation}
			\label{dituishi-intial_data-A_L}
			\begin{cases}
				A_{L+2}+(\sigma h-2)A_{L+1}+A_L=0,\\
				A_1=-1,\ A_2=\sigma h-2,
			\end{cases}
		\end{equation}
		which shows that $A_L$ is a polynomial of $\sigma h$ and the degree of $A_L=A_L(\sigma h)$ is $\deg A_L\leq L-1$. Using the properties of Chebyshev polynomials, it can be known that the polynomial $A_L(\sigma h)$ has $L-1$ distinct positive roots
		\begin{equation*}
			(\sigma h)_j^{L}=2+2\cos\left(\frac{j\pi}{L} \right),\ j=1,2,\cdots,L-1.
		\end{equation*}
		\par		
		For $k\in\ii\R_+$, as $\sigma\to+\infty$, the following limiting holds
		\begin{equation*}
			\lim_{\sigma\to+\infty}a_L(k)=\left(1-\frac{\ii h}{2k} \right)^L,
		\end{equation*}
		which has a root of $L$ order at $k=\frac{\ii h}{2}$.
		\par
		\par	
		The analysis of the above critical cases leads to the proof of the theorem, see Figure \ref{sigma-k-plane} for details.
	\end{proof}
	\par
	From the uniform boundedness of $(\sigma h)_j^{L} (j=1,2,\cdots,L-1)$ with respect to integers $j$ and $L$, one has the following corollary:
	
	\begin{corollary}
		Let $h>0$ and $\sigma>0$, then when $U_n=h$ and $x_n=n\sigma\ (n=1,2,\cdots,L)$ in initial data (\ref{L-delta-intial}), there are $L$ solitons for $\sigma h\geq4$.
	\end{corollary}
	
	Specially, for $L=3$, the following corollary holds immediately:
	
	\begin{corollary} \label{corollary-5-2}
		When $U_1=U_2=U_3=h$, $x_1=\sigma$, $x_2=2\sigma$  and $x_3=3\sigma$ in  initial data (\ref{L-delta-intial}), for the following initial problem
		\begin{equation}
			\begin{cases}
				u_t-6uu_x+u_{xxx}=0,\\
				u(x,0)=-h\delta(x-\sigma)-h\delta(x-2\sigma)-h\delta(x-3\sigma),
			\end{cases}
		\end{equation}
		where $\sigma>0,h>0$, one has the following assertions:
		\begin{itemize}
			\item [(i)] For $0<\sigma h\leq1$, $\sharp_3=1$,
			\item [(ii)] For $1<\sigma h\leq3$, $\sharp_3=2$,
			\item [(iii)] For $3<\sigma h<\infty$, $\sharp_3=3$,
		\end{itemize}
		where $\sharp_3$ represents the number of solitons.
	\end{corollary}
	Finally, numerical simulations are performed for $L=3$ in initial data (\ref{L-delta-intial}) to illustrate how it can evolve into three solitons. In this case, take $h=0.5$ and $\sigma=20$, which satisfies case (iii) in Corollary \ref{corollary-5-2}, then three solitons emerge. Figure \ref{L=3,sig=20,h=0.5,T=100} shows the numerical result of $-u(x,t)$ at $t=100$, where $u(x,t)$ solves the KdV equation with this initial condition. Similar to the case of $L=2$ in Figure \ref{L2,sig=20,h=0.5,T=100}, for time large enough, the asymptotic solution of the KdV equation with delta function initial profile of three spikes can split into dispersive waves and three solitons. It is seen from Figure \ref{L=3,sig=20,h=0.5,T=100} that although the initial amplitude $h$ of the three spikes is the same, as time increases, the amplitudes of the three solitons are different.

	\begin{figure}
		\centering
		\includegraphics[scale=0.4]{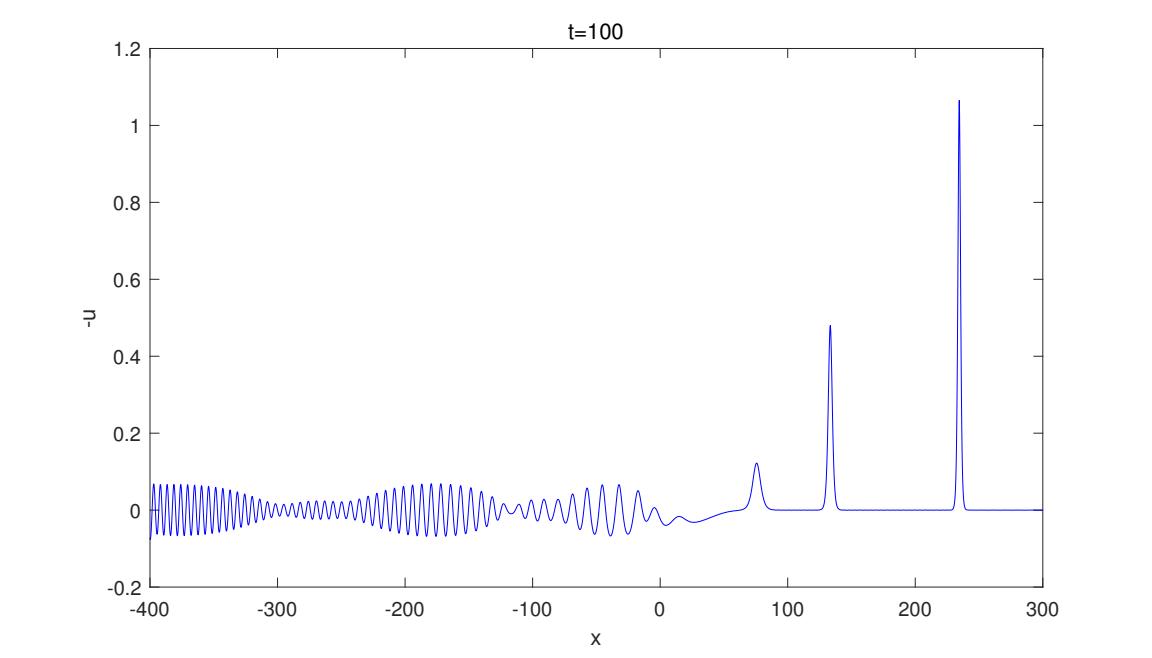}
		\caption{Numerical simulation of $-u(x,t)$ at $t=100$ for $h=0.5$ and $\sigma=20$ in Corollary \ref{corollary-5-2}.}
		\label{L=3,sig=20,h=0.5,T=100}
	\end{figure}

	\appendix
	
	\section{Two model Riemann-Hilbert problems}
	This appendix introduces two model Riemann-Hilbert problems, the parabolic-cylinder model problem \cite{Grunert-Teschl-2009} and the Painlev\'e II model problem.
	\par

	\begin{theorem}[Parabolic-cylinder model problem]
		Denote the contour $\Sigma^{(pc)}=\cup_{j=1}^{4}\Sigma^{(pc)}_j$ in Figure.\ref{pc-contour}, where
		\[\Sigma^{(pc)}_j=\{k\in\C|k=v\ee^{-\frac{\pi}{4}+\frac{\pi}{2}j},v\geq0\}\]
		is shown in Figure.\ref{pc-contour}.
		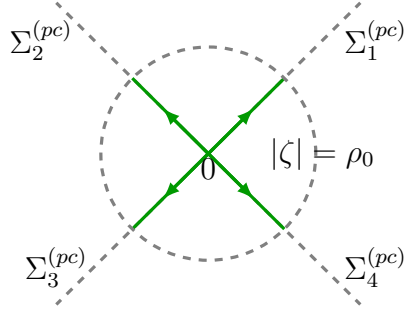
\begin{figure}[htpb]
			\centering
			\begin{tikzpicture}
				\draw[very thick,,dashed,gray](0,0)--(2,2);
				\draw[very thick,,dashed,gray](0,0)--(-2,2);
				\draw[very thick,,dashed,gray](0,0)--(2,-2);
				\draw[very thick,,dashed,gray](0,0)--(-2,-2);
				
				\draw[very thick,,-,black!40!green](0,0)--(1,1);
				\draw[very thick,,-,black!40!green](0,0)--(-1,1);
				\draw[very thick,,-,black!40!green](0,0)--(1,-1);
				\draw[very thick,,-,black!40!green](0,0)--(-1,-1);
				
				\draw[very thick,,-latex,black!40!green](0,0)--(0.6,0.6);
				\draw[very thick,,-latex,black!40!green](0,0)--(-0.6,0.6);
				\draw[very thick,,-latex,black!40!green](0,0)--(0.6,-0.6);
				\draw[very thick,,-latex,black!40!green](0,0)--(-0.6,-0.6);
				
				\draw[very thick,gray,dashed] (0,0) circle [radius=1.414];
				
				\node[]at(0,-0.2){$0$};
				\node[]at(1.5,0){$|\zeta|=\rho_0$};
				\node[]at(2.2,1.5){$\Sigma^{(pc)}_1$};
				\node[]at(-2.2,1.5){$\Sigma^{(pc)}_2$};
				\node[]at(-2,-1.5){$\Sigma^{(pc)}_3$};
				\node[]at(2.2,-1.5){$\Sigma^{(pc)}_4$};
			\end{tikzpicture}
			\caption{The contour $\Sigma^{(pc)}$ for the RH problem of $M^{(pc)}(s,k)$.}
			\label{pc-contour}
		\end{figure}
		
		Find a $2\times2$ matrix-valued function $M^{(pc)}(\zeta)$ satisfying the following RH problem:
		\[\begin{aligned}
			&M^{(pc)}_{+}(\zeta)=M^{(pc)}_{-}(\zeta)V^{(pc)}_j(\zeta),&&\zeta\in\Sigma^{(pc)}_j,\quad j=1,2,3,4,\\
			&M^{(pc)}(\zeta)\to I,&&\zeta\to\infty,
		\end{aligned}\]
		where the jump matrices are
		\[\begin{aligned}
			V^{(pc)}_1(\zeta)=
			\begin{pmatrix}
				1&-r_1(\zeta)\zeta^{2\ii\nu}\ee^{-t\Phi(\zeta)}\\
				0&1
			\end{pmatrix},\quad&
			V^{(pc)}_2(\zeta)=
			\begin{pmatrix}
				1&0\\
				r_2(\zeta)\zeta^{-2\ii\nu}\ee^{t\Phi(\zeta)}&1
			\end{pmatrix},\\
			V^{(pc)}_3(\zeta)=
			\begin{pmatrix}
				1&-r_3(\zeta)\zeta^{2\ii\nu}\ee^{-t\Phi(\zeta)}\\
				0&1
			\end{pmatrix},\quad&
			V^{(pc)}_4(\zeta)=
			\begin{pmatrix}
				1&0\\
				r_4(\zeta)\zeta^{-2\ii\nu}\ee^{t\Phi(\zeta)}&1
			\end{pmatrix}.
		\end{aligned}\]
		
		Assume there is some $\rho_0>0$ such that $V^{(pc)}_j(\zeta)=I$ for $|\zeta|>\rho_0$. Moreover, suppose that within $|\zeta|\leq\rho_0$ the following estimates hold:
		\begin{itemize}
			\item [(i)]The phase function satisfies $\Phi(0)=\ii\Phi_0\in\ii\R$, $\Phi^{\prime}(0)=0$, $\Phi^{\prime\prime}(0)=\ii$ and
			\[\pm\operatorname{Re}(\Phi(\zeta))\geq\frac{1}{4}|\zeta|^2,\quad `+\text{'}~\text{for}\  \zeta\in\Sigma_1\cup\Sigma_3,\quad `-\text{'}~ \text{for}\ \Sigma_2\cup\Sigma_4,\]
			\[|\Phi(\zeta)-\Phi(0)-\frac{\ii\zeta^2}{2}|\leq|\zeta|^3.\]
			\item [(ii)]There is some $r\in\mathbb{D}=\{k\in\C:|k|<1\},\ \alpha\in(0,1]$ and $L\in(0,\infty)$ such that the following H\"{o}lder condition holds:
			\[|r_j(\zeta)-r_j|\leq L|\zeta|^{\alpha},\]
			where $r_j~(j = 1,2,3,4)$ are
			\[r_1=\bar{r},\quad r_2=r,\quad r_3=\frac{\bar{r}}{1-|r|^2}\quad r_4=\frac{r}{1-|r|^2}.\]
		\end{itemize}
		\par
		Then the parabolic-cylinder model problem has a unique solution $M^{(pc)}(\zeta)$ that has the asymptotics for $\zeta\to \infty$:
		\begin{equation}
			M^{(pc)}(\zeta)=I+\frac{1}{\zeta}\frac{\ii}{t^{1/2}}
			\begin{pmatrix}
				0&-\beta\\
				\bar{\beta}&0
			\end{pmatrix}
			+O(t^{-\frac{1+\alpha}{2}}),
		\end{equation}
		for $|\zeta|>\rho_0$, where
		\[\beta=\sqrt{\nu}\ee^{\ii(\pi/4-\arg(r)+\arg(\Gamma(\ii\nu)))}\ee^{-\ii t\Phi_0}t^{-\ii\nu},\quad \nu=\frac{1}{2\pi}\log(1-|r|^2). \]
		\label{PCmodle}
	\end{theorem}
	\begin{theorem}[Painlev\'e II model problem]
		Denote the contour $\Sigma^{(P)}=\cup_{j=1}^{6}\Sigma^{(P)}_j$ in Figure \ref{PII-contour}, where
		\[\Sigma^{(P)}_j=\{k\in\C|k=v\ee^{\frac{\pi}{3}j},v\geq0\},\]
		\[\Omega^{(P)}_j=\{k\in\C|\frac{j-1}{3}\pi<\arg k<\frac{j}{3}\pi\}.\]
		\begin{figure}[H]
			\centering
			\begin{tikzpicture}
				\draw[very thick,black!40!green,-latex](0,0)--(0.577,1);
				\draw[very thick,black!40!green,-](0,0)--(1.154,2);
				\draw[very thick,black!40!green,-latex](0,0)--(-0.577,1);
				\draw[very thick,black!40!green,-](0,0)--(-1.154,2);
				\draw[very thick,black!40!green,-latex](0,0)--(-1.154,0);
				\draw[very thick,black!40!green,-](0,0)--(-2.308,0);
				\draw[very thick,black!40!green,-latex](0,0)--(-0.577,-1);
				\draw[very thick,black!40!green,-](0,0)--(-1.154,-2);
				\draw[very thick,black!40!green,-latex](0,0)--(0.577,-1);
				\draw[very thick,black!40!green,-](0,0)--(1.154,-2);
				\draw[very thick,black!40!green,-latex](0,0)--(1.154,0);
				\draw[very thick,black!40!green,-](0,0)--(2.308,0);
				
				\node[]at(1.3,2.2){$\Sigma^{(P)}_1$};
				\node[]at(-1.3,2.2){$\Sigma^{(P)}_2$};
				\node[]at(-2.5,0){$\Sigma^{(P)}_3$};
				\node[]at(-1.3,-2.2){$\Sigma^{(P)}_4$};
				\node[]at(1.3,-2.2){$\Sigma^{(P)}_5$};
				\node[]at(2.5,0){$\Sigma^{(P)}_6$};
				
				\node[red]at(1.731,1){$\Omega^{(P)}_1$};
				\node[red]at(0,2){$\Omega^{(P)}_2$};
				\node[red]at(-1.731,1){$\Omega^{(P)}_3$};
				\node[red]at(-1.731,-1){$\Omega^{(P)}_4$};
				\node[red]at(0,-2){$\Omega^{(P)}_5$};
				\node[red]at(1.731,-1){$\Omega^{(P)}_6$};
				
				\node[]at(0,-0.5){$0$};
			\end{tikzpicture}
			\caption{The contour $\Sigma^{(P)}$ for the RH problem of $M^{(P)}(s,k)$.}
			\label{PII-contour}
		\end{figure}
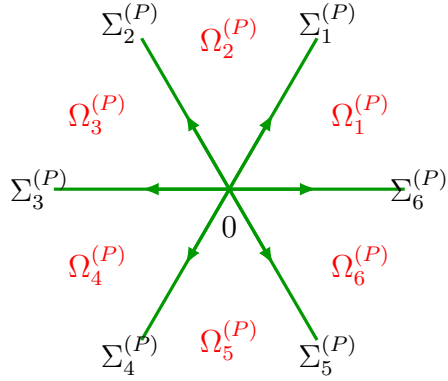
		Assuming $M^{(P)}(s,k)$ satisfies the following RH problem:\\
		(1) $M^{(P)}(s,\cdot)$ is sectionally analytic on $\C\backslash\Sigma^{(P)}$;\\
		(2) $M^{(P)}(s,k)\to I$ as $k\to \infty;$\\
		(3) $M^{(P)}_{+}(s,k)=M^{(P)}_{-}(s,k)V^{(P)}_j$ for $k\in\Sigma^{(P)}_j$, where the jump matrices are
		\begin{equation}
			\begin{aligned}
				&V_1=\begin{pmatrix}
					1&0\\
					p&1
				\end{pmatrix},\quad
				&&V_2=\begin{pmatrix}
					1&r\\
					0&1
				\end{pmatrix},\quad
				&&V_3=\begin{pmatrix}
					1&0\\
					q&1
				\end{pmatrix},\\
				&V_4=\begin{pmatrix}
					1&p\\
					0&1
				\end{pmatrix},\quad
				&&V_5=\begin{pmatrix}
					1&0\\
					r&1
				\end{pmatrix},\quad
				&&V_6=\begin{pmatrix}
					1&q\\
					0&1
				\end{pmatrix},
			\end{aligned}
		\end{equation}
		and $p,q,r$ is independent of $k,s$ and satisfy the constraint
		\[p+q+r+pqr=0.\]
		\par
		Then the Painlev\'e II model problem has a unique solution $M^{(P)}(s,k)$ that has the asymptotics for $k\to \infty$:
		\begin{equation}
			M^{(P)}(s,k)=I+\frac{\ii}{2k}
			\begin{pmatrix}
				\displaystyle\int_{s_0}^{s}y^2(\xi)\dd\xi&-y(s)\\
				y(s)&-\displaystyle\int_{s_0}^{s}y^2(\xi)\dd\xi
			\end{pmatrix}
			+O(k^{-2}),
		\end{equation}
		where $y(s)=y(s;p,q,r)$ is the Painlev\'e II equation with Stokes constants $(p, q, r)$.
		\label{Painleve-II model}
	\end{theorem}
	
	\begin{proof}
		Assuming the $2\times2$ matrix $\Psi=\Psi(s,k)$ satisfies the following linear problem:
		\begin{equation}
			\begin{aligned}
				\frac{\dd\Psi}{\dd k}&=
				\begin{pmatrix}
					-4\ii k^2-\ii s-2\ii y^2&4ku+2\ii y'\\
					4ky-2\ii y'&4\ii k^2+\ii s+2\ii y^2
				\end{pmatrix}\Psi,\\
				\frac{\dd\Psi}{\dd s}&=
				\begin{pmatrix}
					-\ii k &y\\
					y&\ii k
				\end{pmatrix}\Psi,
			\end{aligned}
			\label{PII_CC}
		\end{equation}
		where $y=y(s)$ is the solution of the Painlev\'e II equation
		\begin{equation}
			y''(s)-sy(s)-2y^3(s)=0,
			\label{PII}
		\end{equation}
		with $y'=\dd y/\dd s$. The compatibility conditions of linear problem (\ref{PII_CC}) can deduce the Painlev\'e II equation in (\ref{PII}).
		
		Equation (\ref{PII_CC}) has unique solutions $\Psi_j(s,\cdot)=\Psi(s,\cdot)|_{\Omega^{(P)}_j}$ with the asympototics
		\begin{equation}
			\Psi(s,k)=:\hat{\Psi}(s,k)\ee^{-((4\ii/3)k^3+\ii sk)\sigma_3}:=\left(I+\frac{(\hat{\Psi}_j)_1}{k}+\frac{(\hat{\Psi}_j)_2}{k^2}+\cdots \right)\ee^{-((4\ii/3)k^3+\ii sk)\sigma_3}
			\label{PsiExp}
		\end{equation}
		for $k\in\Omega_j$ as $k\to\infty$. In fact, the function $\Psi(s,k)$ satisfies the jump relations
		\begin{equation}
			\Psi_{j+1}(s,k)=\Psi_j(s,k)V^{(P)}_j,\quad 1\leq j\leq6, \quad(\Psi_7=\Psi_1).
		\end{equation}
		\par
		Substituting the expansion (\ref{PsiExp}) into the linear problem(\ref{PII_CC}) yields
		\begin{equation}
			\hat{\Psi}(s,k)=I+\frac{\ii}{2k}
			\begin{pmatrix}
				\displaystyle\int_{s_0}^{s}y^2(\xi)\dd\xi&-y(s)\\
				y(s)&-\displaystyle\int_{s_0}^{s}y^2(\xi)\dd\xi
			\end{pmatrix}
			+O(k^{-2}).
		\end{equation}
		
		By using the uniqueness of the RH problem, it is known that $M^{(P)}(s,k)=\hat{\Psi}(s,k)$. This completes the proof of this theorem.
	\end{proof}
	
	\subsection*{Acknowledgements}
	
	This work was funded by National Natural Science Foundation of China through grant numbers 12371247 and 12431008.


\begin{thebibliography}{99}
		\footnotesize\itemsep=0pt
		\providecommand{\eprint}[2][]{\href{http://arxiv.org/abs/#2}{arXiv:#2}}
		
		
		\bibitem{KdV-1895}D. J. Korteweg, G. de Vries, On the change of a form of long waves advancing in a rectangular canal and a new type of long stationary waves, London, Edinburgh Dublin Philos. Mag. J. Sci. 39 (1895) 422-443.
		
		\bibitem{GGKM-1965}C. S. Gardner, J. M. Greene, M. D. Kruskal, R. M. Miura, Method for solving the Korteweg-de Vries equation, Phys. Rev. Lett. 19(19) (1967) 1095.
		
		
		\bibitem{Lax-1968}P. D. Lax, Integrals of nonlinear equations of evolution and solitary waves, Commun. Pure Appl. Math. 21(5) (1968) 467-490.
		
		\bibitem{Drazin-1989}P. G. Drazin, R. S. Johnson, Solitons: an introduction, Cambridge university press, Vol. 2, (1989).
		
		\bibitem{Marchenko-2011}V. A. Marchenko, Sturm-Liouville operators and applications, American Mathematical Soc., Vol. 373, (2011).
		
		\bibitem{Egorova-Teschl-2015}I. Egorova, Z. Gladka, T. L. Lange, G. Teschl, Inverse Scattering Theory for Schr\"{o}dinger Operators with Steplike Potentials, Math. Phys. Anal. Geom. 11(2) (2015) 123-158.
		
		
		\bibitem{Ablowitz-Segur-B-1981}M. Ablowitz and H. Segur, Solitons and Inverse Scattering Transform (SIAM, Philadelphia, 1981), Vol. 4.
		
		
		\bibitem{Ablowitz-Luo-2018}M. J. Ablowitz, X. D. Luo, J. T. Cole, Solitons, the Korteweg-de Vries equation with step boundary values, and pseudo-embedded eigenvalues, J. Math. Phys. 59(9) (2018) 091406.
		
		
		\bibitem{Ablowitz-Luo-2023}M. J. Ablowitz, J. T. Cole, G. A. El, M. A. Hoefer, X. D. Luo, Soliton-mean field interaction in Korteweg-de Vries dispersive hydrodynamics, Stud. Appl. Math. 151(3) (2023) 795-856.
		
		\bibitem{Ablowitz-Kodama-1982}M. J. Ablowitz, Y. Kodama, Note on asymptotic solutions of the Korteweg-de Vries equation with solitons,Stud. Appl. Math. 66 (1982) 159-170.
		
		\bibitem{Claeys-Grava-2010}T. Claeys, T. Grava, Solitonic asymptotics for the Korteweg-de Vries equation in the small dispersion limit, SIAM J. Math. Analy. 42(5) (2010) 2132-2154.
		
		\bibitem{Fokas-Lenells-2010}A. S. Fokas, J. Lenells, Explicit soliton asymptotics for the Korteweg-de Vries equation on the half-line, Nonlinearity, 23(4) (2010) 937.
		
		\bibitem{Fokas-Lenells-2010}A. S. Fokas, J. Lenells, Explicit soliton asymptotics for the Korteweg-de Vries equation on the half-line, Nonlinearity, 23(4) (2010) 937.
		
		\bibitem{Its-Sukhanov-2020} A. Its, V. Sukhanov, Large time asymptotics for the cylindrical Korteweg-de Vries equation. I. Nonlinearity 33 (2020) 5215.
		
		\bibitem{Girotti-Grava-2021}M. Girotti, T. Grava, R. Jenkins, K. T. R. McLaughlin, Rigorous asymptotics of a KdV soliton gas, Commun. Math. Phys. 384 (2021) 733-784.
		
		\bibitem{Hruslov-1976}E. J. Hruslov, Asymptotics of the solution of the Cauchy problem for the Korteweg-de Vries equation with initial data of step type, Math. USSR-Sbornik, 28(2) (1976) 229.
		
		\bibitem{Cohen-1984}A. Cohen, Solutions of the Korteweg-de Vries equation with steplike initial profile,
		Comm. Partial Differ. Equ. 9 (1984) 751-806.
		
		\bibitem{Deift-Zhou-1993}P. Deift, X. Zhou, A steepest descent method for oscillatory Riemann-Hilbert problems. Asymptotics for the MKdV equation, Ann. Math. 137(2) (1993) 295-368.
		
		\bibitem{Deift-2017}P. Deift, Some open problems in random matrix theory and the theory of integrable systems. II. SIGMA Symmetry Integrability Geom. Methods Appl. 13 (2017), Paper No. 016, 23 pp.
		
		\bibitem{Deift-Venakides-Zhou-1994}P. Deift, S. Venakides, X. Zhou, The collisionless shock region for the long-time behavior of solutions of the KdV equation, Commun. Pure Appl. Math. 47(2) (1994) 199-206.
		
		\bibitem{Grunert-Teschl-2009}K. Grunert, G. Teschl, Long-time asymptotics for the Korteweg-de Vries equation via nonlinear steepest descent, Math. Phys. Anal. Geom. 12(3) (2009) 287-324.
		
		\bibitem{Teschl-Nonlinearity-2013}I. Egorova, Z. Gladka, V. Kotlyarov,
		G. Teschl, Long-time asymptotics for the Korteweg-de Vries equation with step-like initial data, Nonlinearity 26 (2013) 1839-1864.
		
		\bibitem{Teschl-JDE-2016}K. Andreiev, I. Egorova, T. L. Lange, G. Teschl, Rarefaction waves of the Korteweg-de Vries equation via nonlinear steepest descent, J. Diff. Equations, 261(10) (2016) 5371-5410.
		
		\bibitem{Trogdon-Olver-Deconinck-2012}T. Trogdon, S. Olver, B. Deconinck, Numerical inverse scattering for the Korteweg-de Vries and modified Korteweg-de Vries equations, Physica D 241(11) (2012) 1003-1025.
		
		\bibitem{Trogdon-Deconinck-2014}T. Trogdon, B. Deconinck, A numerical dressing method for the nonlinear superposition of solutions of the KdV equation, Nonlinearity 27 (2014) 67.
		
		\bibitem{Bilman-Trogdon-2020}D. Bilman, T. Trogdon, On numerical inverse scattering for the Korteweg-de Vries equation with discontinuous step-like data, Nonlinearity, 33(5) (2020) 2211.
		
		\bibitem{Cohen-Kappeler-1987}A. Cohen, T. Kappeler, The asymptotic behavior of solutions of the korteweg-deVries equation evolving from very irregular data, Annals of Physics 178.1 (1987) 144-185.
		
		\bibitem{Atkinson-1975}D. A. Atkinson, H. W. Crater, An exact treatment of the Dirac delta function potential in the Schr\"{o}dinger equation, Amer. J. Phys. 43(4) (1975) 301-304.
		
		\bibitem{Aktosun-1999} T. Aktosun, On the Schr\"{o}dinger equation with steplike potentials, J. Math. Phys. 40(11) (1999) 5289-5305.
		
		\bibitem{Buslaev-Fomin-1962} V. S. Buslaev, V. N. Fomin, An inverse scattering problem for the one-dimensional Schr\"{o}dinger equation on the entire axis, Vestnik Leningrad Univ. 17 (1962) 56-64.
		
		\bibitem{Cohen-Kappeler-1985} A. Cohen, T. Kappeler, Scattering and inverse scattering for steplike potentials in the Schr\"{o}dinger equation, Indiana Univ. Math. J. 34 (1985) 127-180.
		
		\bibitem{Deift-Park-2011}P. Deift, J. Park, Long-time asymptotics for solutions of the NLS equation with a delta potential and even initial data, Int. Math. Res. Not. 2011 (2011), 5505-5624.
		
		\bibitem{Miura-1968}R. M. Miura, Korteweg-de Vries equation and generalizations. I. A remarkable explicit nonlinear transformation, J. Math. Phys. 9(8) (1968) 1202-1204.
		
		\bibitem{Cohen-1979}A. Cohen, Existence and regularity for solutions of the Korteweg-de Vries equation, Arch. Ration. Mech. Anal. 71 (1979) 143-175.
		
		\bibitem{Bilman-Deift-Kriecherbauer-Mclaughlin-Nenciu-2024} D. Bilman, P. Deift, T. Kriecherbauer, K. Mclaughlin, I. Nenciu, Long-time asymptotics for the Korteweg-de Vries equation, in preparation, (2024).		
		
		\bibitem{Segur-Ablowitz-1981}H. Segur, M. J. Ablowitz, Asymptotic solutions of nonlinear evolution equations and a Painlev\'e transcendent, Phys. D 3 (1981) 165-184.	
		
		\bibitem{Beals-Coifman-1984}R. Beals, R. Coifman, Scattering and inverse scattering for first order systems, Comm. Pure Appl. Math. 37 (1984) 39-90.
		
		
		
	\end{thebibliography}
	
	\pdfbookmark[1]{References}{ref}
	
	
\end{document}